\documentclass[12pt]{amsart}
\usepackage{amssymb,amsmath,amscd,graphicx,
latexsym,amsthm}
\usepackage{amssymb,latexsym,eufrak,amsmath,amscd,graphicx}
\usepackage[mathscr]{eucal}
  \usepackage[all]{xy}
\theoremstyle{plain}
\newtheorem{theorem}{Theorem}[section]

\newtheorem{proposition}[theorem]{Proposition}
\newtheorem{corollary}[theorem]{Corollary}
\newtheorem{lemma}[theorem]{Lemma}

\numberwithin{equation}{section}

\newtheorem{theoremalpha}{Theorem}
\newtheorem{corollaryalpha}[theoremalpha]{Corollary}
\newtheorem{propositionalpha}[theoremalpha]{Proposition}

\theoremstyle{definition}
\newtheorem{definition}[theorem]{Definition}

\newtheorem{remark}[theorem]{Remark}
\newtheorem{example}[theorem]{Example}
\newtheorem{question}[theorem]{Question}

\newtheorem{problem}[theorem]{Problem}

\newcommand{\lra}{\longrightarrow}
\newcommand{\noi}{\noindent}
\newcommand{\PP}{\mathbf{P}}
\newcommand{\PPP}{\PP_{\text{sub}}}
\newcommand{\RR}{\mathbf{R}}

\newcommand{\NN}{\mathbf{N}}
\newcommand{\ZZ}{\mathbf{Z}}
\newcommand{\CC}{\mathbf{C}}
\newcommand{\QQ}{\mathbf{Q}}

\newcommand{\OO}{\mathcal{O}}
\newcommand{\cO}{\mathcal{O}}

\newcommand{\fra}{\mathfrak{a}}
\newcommand{\frb}{\mathfrak{b}}

\newcommand{\bull}{_{\bullet}}
\newcommand{\bulll}[1]{_{{\bullet}, {#1}}}

\newcommand{\frakm}{\mathfrak{m}}

\newcommand{\KK}{\mathbf{K}}

\newcommand{\eps}{\varepsilon}
\newcommand{\sbl}{\vskip 3pt}
\newcommand{\lbl}{\vskip 6pt}

\newcommand{\HH}[3]{H^{{#1}} \big( {#2} , {#3}
\big) }
\newcommand{\HHH}[3]{H^{{#1}} \Big( {#2} , {#3}
\Big) }
\newcommand{\hh}[3]{h^{{#1}} \big( {#2} , {#3}
\big) }

\newcommand{\Effbar}{\overline{\textnormal{Eff}}}

\newcommand{\ord}{\textnormal{ord}}
\newcommand{\codim}{\textnormal{codim}}
\newcommand{\mult}{\textnormal{mult}}
\newcommand{\pr}{\prime}

\newcommand{\num}{ \equiv_{\text{num}} }
\newcommand{\lin}{\equiv_{\text{lin}} }
\newcommand{\vol}{\textnormal{vol}}

\newcommand{\dra}{\dashrightarrow}
\newcommand{\Bl}{\text{Bl}}

\newcommand{\CCC}{\textnormal{closed convex cone}}

\newcommand{\Linser}[1]{| \mspace{1.5mu} {#1}
\mspace{1.5mu} |}
\newcommand{\linser}[1]{\Linser{  {#1}  }}

\newcommand{\Bs}[1]{\textnormal{Bs}\big(
\mspace{1.1mu} | {#1} | \mspace{1.1mu} \big)}

\newcommand{\tn}[1]{\textnormal{{#1}}}
\newcommand{\Image}{\textnormal{Im}}
\newcommand{\Supp}{\textnormal{supp}}
 \newcommand{\pro}{\textnormal{pr}}
 \newcommand{\vecbull}{_{\Vec{\bullet}}}
 \newcommand{\interior}{\textnormal{int}}
\newcommand{\BBB}{\mathbf{B}}

\begin{document}

\title{Convex bodies associated to linear series}

\author{Robert Lazarsfeld}
  \address{Department of Mathematics, University of Michigan, Ann Arbor, MI
   48109}
 \email{{\tt rlaz@umich.edu}}
 \thanks{Research of the first author partially supported by NSF grant DMS-0652845.}

\author{Mircea Musta\c   t\v a}
  \address{Department of Mathematics, University of Michigan, Ann Arbor, MI
  48109}
 \email{{\tt mmustata@umich.edu}}
 \thanks{Research of the second author partially supported by NSF grant DMS-0500127 and a Packard Fellowship.}

\maketitle

\setlength{\parskip}{0 in}

\tableofcontents
\setlength{\parskip}{.1 in}

\section*{Introduction}
In his interesting papers \cite{Okounkov96} and \cite{Okounkov03}, Okounkov showed in passing that one could associate a convex body to a linear series on a projective variety, and then use convex geometry to study such linear systems. Although Okounkov was essentially working in the classical setting of ample line bundles, it turns out that the  construction goes through for an arbitrary big divisor. Moreover, one can recover and extend from this viewpoint most of the fundamental results  from the asymptotic theory of linear series. The purpose of this paper is to initiate a systematic development of this theory, and to give a number of applications and examples.

We start by describing Okounkov's construction. Let $X$ be a smooth irreducible  projective variety of dimension $d$ defined over an uncountable algebraically closed field $\KK$ of arbitrary characteristic.\footnote{In the body of the paper, we will relax many of the hypotheses appearing here in the Introduction.}  The construction depends upon the choice of  a fixed flag
\[  Y\bull  \ \  : \ \ X \, = \, Y_0 \  \supseteq \ Y_1
\
\supseteq Y_2 \ \supseteq \  \ldots \ \supseteq \  Y_{d-1} \
\supseteq
\ Y_d \, =
\, \{ \tn{pt} \},  
\]
where $Y_i$ is a smooth irreducible subvariety of codimension $  i$ in $X$.  Given a big divisor\footnote{Recall that by definition a divisor $D$ is \textit{big} if $\hh{0}{X}{\OO_X(mD)}$ grows like $m^d$.} $D$ on $X$, one defines a valuation-like function
\begin{equation}  
\nu = \nu_{Y\bull} =  \nu_{Y\bull, D} \, : \,  \big (\HH{0}{X}{\OO_X(D)} - \{0\} \big)    \lra \ZZ^d   \ \ , \ \ s \mapsto   \nu(s) \, = \, \big(
\nu_1(s) , \ldots, \nu_d(s) \, \big)
\tag{*} \end{equation}
as follows. First,   set $\nu_1 = \nu_1(s) = \ord_{Y_1}(s)$. Then $s$ determines in the natural way a section \[ \tilde{s}_1 \in \HH{0}{X}{ \OO_X(D- \nu_1 Y_1)}\] that does not vanish identically along $Y_1$, and
so we get by restricting a non-zero section
\[
s_1 \ \in \ \HH{0}{Y_1}{\OO_{Y_1}(D - \nu_1 Y_1)}. \]
Then take 
\[ \nu_2(s)  \ = \ \ord_{Y_2}(s_1), \]
and continue in this manner to define the remaining  $\nu_i(s)$. For example, when $X = \PP^d$ and $Y{\bull}$ is a flag of linear spaces, $\nu_{Y\bull}$ is  essentially the lexicographic valuation on polynomials. 

	Next,  define
\[  v(D) \ =  \ \tn{Im} \big(  \, (\HH{0}{X}{\OO_X(D)} - \{ 0 \})   \overset {\nu_Y} \lra \ZZ^d \,  \big) \]
to be the set of valuation vectors of non-zero sections of $\OO_X(D)$.
 It is not hard to check that
\begin{equation}\label{lattice.pt.count}  \# \,  v(D) \ = \ h^0(X, \OO_X(D)). \notag \end{equation}
Then finally set
\begin{equation}\label{Delta.Defn}
\Delta(D) \ = \ \Delta_{Y\bull}(D) \ = \ \ \text {closed convex hull} \Big(\bigcup_{m \ge 1} \tfrac{1}{m} \cdot v(mD) \Big). \notag \end{equation}
Thus $\Delta(D)$ is a convex body in $\RR^d = \ZZ^d \otimes \RR$, which we  call the \textit{Okounkov body} of $D$ (with respect to the fixed flag $Y\bull$).  

 One can view Okounkov's construction as a generalization of a correspondence familiar from toric geometry, where a torus-invariant divisor $D$ on a toric variety $X$ determines a rational polytope $P_D$. In this case, working with respect to a flag of invariant subvarieties of $X$,  $\Delta(D)$ is a translate of $P_D$. An analogous polyhedron on spherical varieties has been studied in \cite{Brion}, \cite{Okounkov97}, \cite{Alexeev.Brion}, \cite{Kaveh}.
 On the other hand, the convex bodies $\Delta(D)$ typically have a less classical flavor even when $D$ is ample. For instance, 
 let $X$ be an abelian surface having Picard number $\rho(X) \ge 3$, and choose an ample curve $C\subseteq X$ together with a smooth point $x \in C$, yielding the flag \[ X \ \supseteq \ C \ \supseteq  \ \{x\}. \] Given an ample divisor $D$ on $X$, denote by $\mu = \mu(D) \in \RR$ the smallest root of the quadratic polynomial $p(t) = (D - t C)^2 $: for most choices of $D$,  $\mu(D)$ is irrational.  Here the Okounkov body of $D$  is the trapezoidal region in $\RR^2$ shown in Figure \ref{Okounkov.Body.Picture}. Note that  in this case $\Delta(D)$, although polyhedral, is usually not rational.  We give in \S \ref{Round.Cone.Ex} a four-dimensional example where $\Delta(D)$ is not even polyhedral.
 \begin{figure} \label{Intro.Picture}
\includegraphics[scale = .5]{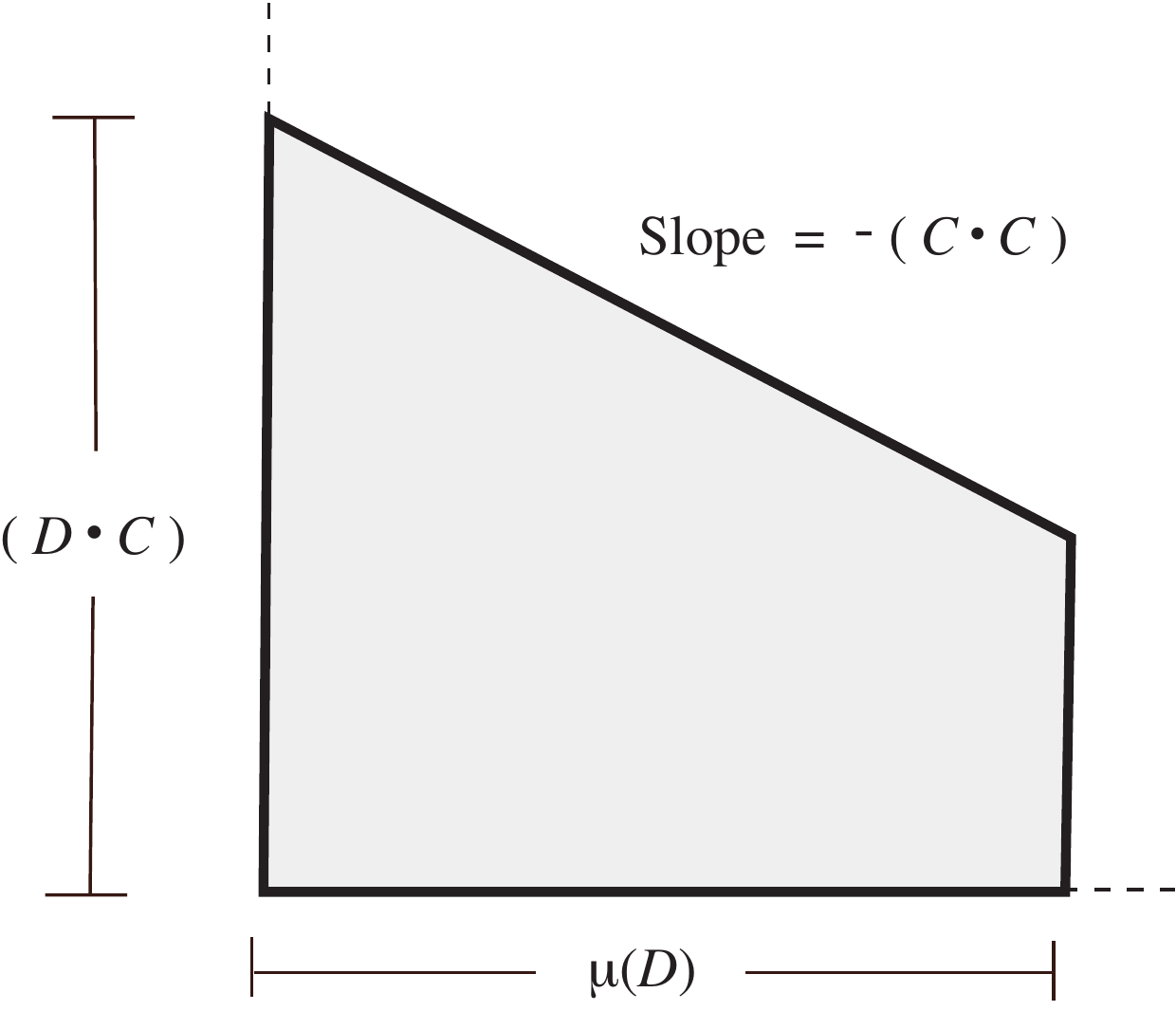}
\caption{Okounkov body of a divisor on  an abelian surface}
 \label{Okounkov.Body.Picture}
\end{figure}

As one might suspect, the standard Euclidean volume of $\Delta(D)$ in $\RR^d$ is related to the rate of growth of the groups $\hh{0}{X}{\OO_X(mD)}$. In fact,  Okounkov's arguments in \cite[\S 3]{Okounkov03} -- which are based on results \cite{Khovanskii} of Khovanskii -- go  through without change to prove 
\begin{theoremalpha} \label{Vol.Okounkov.Body.Eqn} If $D$ is any big divisor on $X$, then 
\[ \vol_{\RR^d}\big( \Delta(D) \big) \ = \ \frac{1}{d!} \cdot \vol_X(D).\] \end{theoremalpha}
\noi The quantity on the right is the \textit{volume} of $D$, defined as the limit \[  \vol_X(D) \ =_{\text{def}} \ \lim_{m \to \infty} \frac{ \hh{0}{X}{\OO_X(mD)}}{m^d / d!} .\]
In the classical case, when $D$ is ample, $\vol_X(D) = \int c_1(\OO_X(D))^d $ is just the top self-intersection number of $D$.  In general, the volume is an interesting and delicate invariant of a big divisor, which has lately  been the focus of considerable work (cf. \cite[Chapt.  2.2]{PAG}, \cite{Boucksom}, \cite{ELMNP2}). It plays a pivotal role in several important recent developments in higher dimensional geometry, eg \cite{BDPP}, \cite{Tsuji}, \cite{HM}, \cite{Tak}.

We study the variation of these bodies as functions of $D$. It is not hard to check that $\Delta(D)$ depends only on the numerical equivalence class of $D$, and that $\Delta(pD) = p \cdot \Delta(D)$ for every positive integer $p$. It follows that there is a naturally defined Okounkov body $\Delta(\xi) \subseteq \RR^d$ associated to every rational numerical equivalence class $\xi \in N^1(X)_{\QQ}$, and as before
$
\vol_{\RR^d}(\Delta(\xi)) = \frac{1}{d!} \cdot \vol_X(\xi)$. 
 We prove:
\begin{theoremalpha} \label{Global.Body.Intro.Theorem}
There exists a closed convex cone 
\[ \Delta(X) \     \subseteq \ \RR^d \times N^1(X)_{\RR}\] characterized by the property that in the diagram 
 \[
\xymatrix{
\Delta(X)\ar[dr]      & \subseteq & \RR^d \times N^1(X)_{\RR} \ar[dl]^{\tn{pr}_2}  \\
&  N^1(X)_{\RR},
}
\]
 the fibre 
 $\Delta(X)_\xi \subseteq \RR^d \times\{ \xi \} = \RR^d$ of $\Delta(X)$ over any big class $\xi \in N^1(X)_{\QQ}$ is $\Delta(\xi) $. 
\end{theoremalpha}
\noi This is illustrated schematically in Figure \ref{Global.Body.Picture}. The image of $\Delta(X)$ in $N^1(X)_{\RR}$ is the so-called pseudo-effective cone $\overline {\tn{Eff}}(X)$ of $X$, i.e. the closure of the cone spanned by all effective divisors: its interior is the big cone $\tn{Big}(X)$ of $X$. Thus the theorem yields a natural definition of  $\Delta(\xi) \subseteq \RR^d$ for any big class $\xi \in N^1(X)_{\RR}$, viz. $\Delta(\xi) = \Delta(X)_{\xi}$.  It is amusing to note that already  in the example of an abelian surface considered above, the cone $\Delta(X)$ is non-polyhedral.\footnote{This follows for instance from the observation that $\mu(D)$ varies non-linearly in $D$.} 
 
\begin{figure}  
\includegraphics[scale = .6]{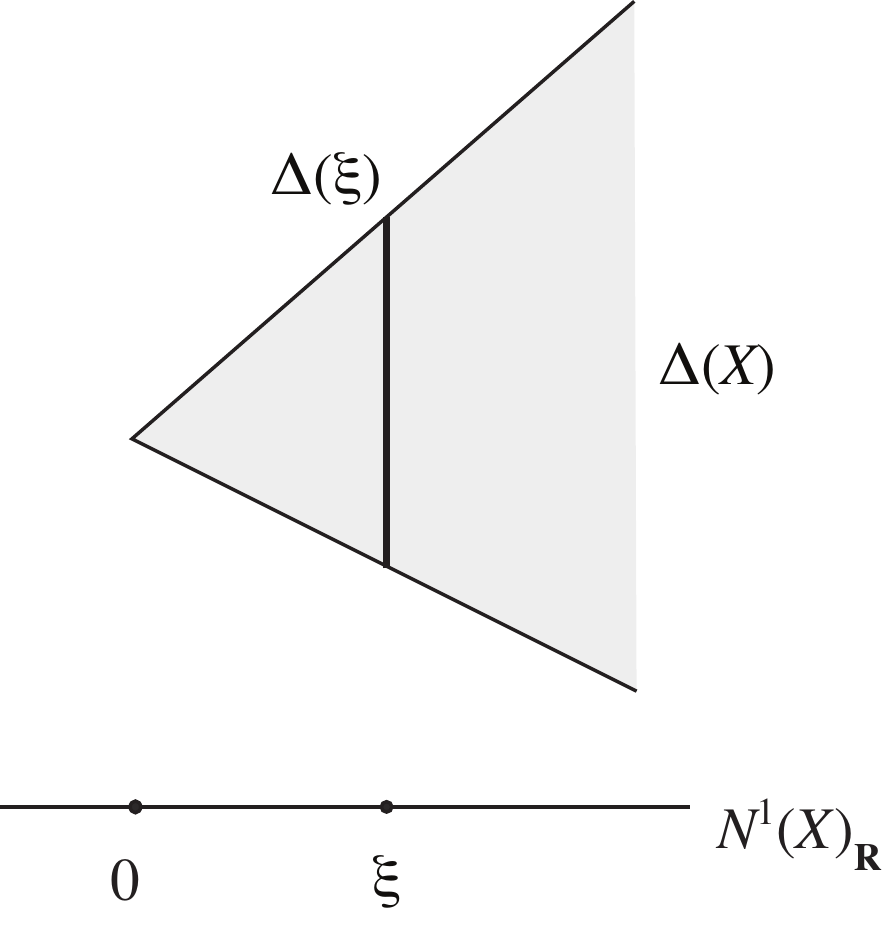}  
\caption{Global Okounkov body}
\label{Global.Body.Picture}
\end{figure}

Theorem \ref{Global.Body.Intro.Theorem} renders transparent several basic properties of the volume function $\vol_X$ established by the first author in \cite[2.2C, 11.4.A]{PAG}. First, since the volumes of the fibres $\Delta(\xi) = \Delta(X)_\xi$ vary continuously for $\xi$ in the interior of  $\pro_2(\Delta(X)) \subseteq \NN^1(X)_\RR$, one deduces that the volume of a big class is computed by a continuous function \[
\vol_X : \tn{Big}(X)  \lra \RR.
\]
Moreover $\Delta(\xi) + \Delta(\xi^\pr) \ \subseteq \ \Delta(\xi + \xi^\pr)$ for any two big classes $\xi, \xi^\pr \in N^1(X)_{\RR}$, and so the Brunn-Minkowski theorem yields the log-concavity relation
\[ 
\vol_X(\xi + \xi^\pr)^{1/d}  \ \ge \ \vol_X(\xi)^{1/d} + \vol_X(\xi^\pr)^{1/d}
\] 
for any two such classes.\footnote{In the classical setting, it was this application of Brunn-Minkowski that motivated Okounkov's construction in \cite{Okounkov03}. We remark that it was established in \cite{PAG} that $\vol_X$ is actually continuous on all of $N^1(X)_{\RR}$ -- i.e. that $\vol_X(\xi) \to 0$ as $\xi$ approaches the boundary of  the pseudo-effective cone $ \overline {\tn{Eff} }(X) $ -- but this doesn't seem to  follow directly from the present viewpoint. The continuity of volumes on compact complex manifolds was proven by Boucksom in \cite{Boucksom}, who works in fact with arbitrary $(1,1)$-classes.}

The Okounkov construction also reveals some interesting facts about the volume function that had not been known previously. For instance, let $E \subseteq X$ be a very ample divisor on $X$ that is general in its linear series, and choose the flag $Y\bull$ in such a way that $ Y_1 = E$. Now  construct  the Okounkov body $\Delta(\xi) \subseteq \RR^d$ of any big class $\xi \in \tn{Big}(X)$, and consider the mapping
\[
\pro_1 : \Delta(\xi) \lra \RR\]
obtained via the projection $\RR^d \lra \RR$ onto the first factor, so that $\pro_1$ is ``projection onto the $\nu_1$-axis." Write $e \in N^1(X)$ for the class of $E$, and given $t > 0$ such that $\xi - te$ is big, set
\[  \Delta(\xi)_{\nu_1 =t} \ = \ \pro_1^{-1}(t) \ \subseteq \ \RR^{d-1} 
\ \ , \ \ \Delta(\xi)_{\nu_1 \ge t} \ =  \ \pro_1^{-1}\big( [t, \infty) \big) \ \subseteq\ \RR^d. \]
 We prove that
\begin{gather*}
\Delta(\xi)_{\nu_1\ge t} \ =_{\text{up to translation}}  \Delta(\xi - te) \\
\vol_{\RR^{d-1}}\big( \Delta(\xi)_{\nu_1 = t} \big)  \ = \ \frac{1}{(d-1)!} \cdot \vol_{X|E}(\xi - te).
\end{gather*}
Here $\vol_{X|E}$ denotes the restricted volume function from $X$ to $E$ studied in \cite{ELMNP3}: when $D$ is integral, $\vol_{X|E}(D)$ measures the rate of growth of the subspaces of  $\HH{0}{E}{\OO_E(mD)}$ consisting of sections that come from $X$. The situation  is illustrated in Figure \ref{Slices.Picture}. 
\begin{figure}  
\includegraphics[scale = .6]{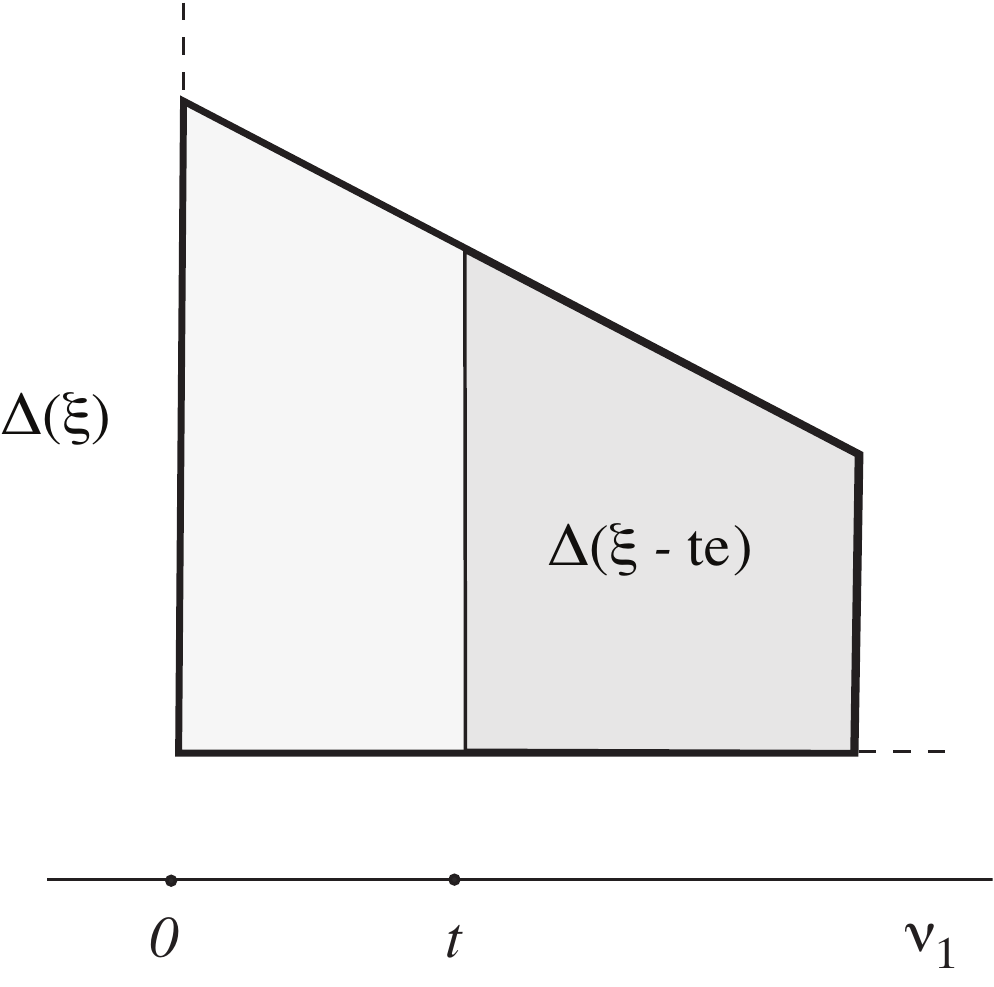}  
\caption{Slices of Okounkov body}
\label{Slices.Picture}
\end{figure}
Since one can compute the $d$-dimensional volume of $\Delta(\xi)$ by integrating the $(d-1)$-dimensional volumes of its slices, one finds:
\begin{corollaryalpha} \label{Integration.Intro}
Let $a >0$ be any real number such that $\xi - ae \in \tn{Big}(X)$. Then 
\[
\vol_X(\xi) - \vol_X(\xi - ae) \ = \ d \cdot \int_{-a}^{0} \vol_{X|E}(\xi + te) \,  dt.
\]
Consequently, the function $t \mapsto \vol_X(\xi + te)$ is differentiable at $t = 0$, and
\[
\frac{d}{dt} \,\big( \vol_X(\xi  + te) \big) |_{t = 0} \ = \ d \cdot \vol_{X|E}(\xi).
\]
\end{corollaryalpha}
\noi This leads to the fact that $\vol_X$ is $\mathcal{C}^1$ on $\tn{Big}(X)$. Corollary \ref{Integration.Intro}
 was one of the starting points of the interesting work \cite{BFJ} of Boucksom--Favre--Jonsson, who found a nice formula for the derivative of $\vol_X$ in any direction, and used it to answer some questions of Teissier.

 Okounkov's construction works   for incomplete as well as for complete linear series. Recall that a \textit{graded linear series} $W\bull$ associated to a big divisor $D$ on $X$ consists of subspaces 
 \[ W_m  \ \subseteq  \ \HH{0}{X}{\OO_X(mD)}\] satisfying the condition that $R(W\bull) = \oplus \, W_m$ be a graded subalgebra of the section ring $R(D) = \oplus \, \HH{0}{X}{\OO_X(mD)}$. These arise naturally in a number of situations (cf \cite[Chapter 2.4]{PAG}), and under mild hypotheses, one can attach to $W\bull$ an Okounkov body 
 \[ \Delta(W\bull)  \ = \  \Delta_{Y\bull}(W\bull) \ \subset \  \RR^d.\]  We use these to extend   several results  hitherto known only for global linear series. For example, we  prove  a version of the Fujita approximation theorem for graded linear series:
 \begin{theoremalpha} \label{FAT.Intro}
 Assume that the rational mapping defined by  $\linser{W_m}$ is birational onto its image for all $m \gg 0$, and fix $\eps > 0$. There exists an integer $p_0 = p_0(\eps)$ having the property that if $p \ge p_0$ then
 \[
 \lim_{k \to \infty} \frac
 { \dim \Image \big( S^k W_p \lra W_{kp}\big)}{p^dk^d/d!} \ \ge \ \vol(W\bull) - \eps.
 \]
 \end{theoremalpha}
\noi When $W_m = \HH{0}{X}{\OO_X(mD)}$ is the complete linear series of a big divisor $D$, this  implies a basic theorem of Fujita (\cite{Fujita}, \cite{DEL},  \cite{Nakamaye}, \cite{Takagi}, \cite[Chapter 11.4]{PAG}) to the effect that the volume of $D$ can be approximated arbitrarily closely by the self-intersection of an ample divisor on a modification of $X$. As an application of Theorem \ref{FAT.Intro}, we give a new proof of a result  \cite{Mustata}  of the second author concerning multiplicities of graded families of ideals, and extend it to possibly singular varieties. We also prove for graded linear series an analogue of Theorem \ref{Global.Body.Intro.Theorem}, which leads to transparent new proofs of several  of the results of \cite{ELMNP3} concerning restricted volumes.  
 
Returning to the global setting, recall that Okounkov's construction depends upon picking a flag $Y\bull$ on $X$. We show that one can eliminate this non-canonical choice by working instead with generic infinitesimal data. Specifically, fix a smooth point $x \in X$, and a complete flag $V\bull$ of subspaces
 \[
 T_xX \ = \ V_0 \ \supseteq \ V_1 \ \supseteq \ V_2 \ \supseteq \ \ldots \ \supseteq \ V_{d-1} \ \supseteq \ \{ 0 \}
 \] 
 in the tangent space to $X$ at $x$. Consider the blowing up 
$ \mu : X^\pr = \Bl_x(X) \lra X$
of $X$ at $x$, with exceptional divisor $E= \PP(T_xX)$. Then the projectivizations of the $V_i$ give rise in the evident manner to a flag $F\bull = F(x; V\bull)$ of subvarieties of $X^\pr$. On the other hand, for a big divisor $D$  on $X$, write $D^\pr = \mu^* D$ and  note that
 \[ \HH{0}{X}{\OO_X(mD)} = \HH{0}{X^\pr}{\OO_{X^\pr}(mD^\pr)} \]
 for all $m$. So it is natural to define $\Delta_{F\bull}(D)$ to be the Okounkov body of $D^\pr$ computed on $X^\pr$ with respect to the flag $F\bull$.   \begin{propositionalpha}
For very general choices of $x$ and $V\bull$, the Okounkov bodies \[      \Delta_{F(x; V\bull)}(D) \ \subseteq \ \RR^d \]
all coincide. In particular, the resulting convex body $\Delta^\pr(D) \subseteq \RR^d$ is canonically defined.
 \end{propositionalpha}
 \noi  Similarly there is a global cone
$\Delta^\pr(X)  \subseteq \ \RR^d \times N^1(X)_{\RR}$
 that does not depend on any auxiliary choices. We suspect that these carry interesting geometric information, but unfortunately they seem very hard to compute. We hope to return to this at a later date. The Proposition follows from a general result about the varying the flag in the Okounkov construction.

As the preparation of this paper was nearing completion, the interesting preprint \cite{KK} of Kaveh and Khovanskii appeared. Those authors study essentially the same construction as here, but from a rather different viewpoint. Starting (as did Okounkov)  with a finite dimensional subspace $L \subseteq \KK(X)$ of the field of rational functions on a variety $X$, Kaveh and Khovanskii  associate to $L$  a convex body $\Delta(L)$ depending on the choice of a valuation $\nu : \KK(X)^* \lra \ZZ^d$. They then relate  geometric invariants of these bodies   to some intersection-theoretic quantities that they define and study. They use the resulting correspondence between convex and algebraic geometry to give new proofs of some basic results on each side. 

Concerning the organization of the paper, we start in \S \ref{Okounkov.Construction.Section} by defining the Okounkov bodies attached to big divisors or graded linear systems. We observe in Proposition \ref{Realize.Convex.Body.Prop} that up to translation and scaling every convex body is realized as the Okounkov body of a graded linear series on projective space. Section \ref{Volume.Ok.Body.Section}
 is devoted to the proof of Theorem \ref{Vol.Okounkov.Body.Eqn} and to some conditions that lead to the corresponding statement for graded linear series. We show in \S  \ref{RLSS} that these conditions are satisfied in the important case of restricted linear series. In \S \ref{Fujita.Approximation.Section} we turn to Fujita's approximation theorem: we give a new proof of the classical result, establish  its extension Theorem \ref{FAT.Intro}, and present the application to graded systems of ideals. Section \ref{Varation.of.OB.Section} revolves around the variational theory of Okounkov bodies: we prove Theorem \ref{Global.Body.Intro.Theorem}
 and its extension to $\NN^r$-graded linear series, and establish Corollary \ref{Integration.Intro}. The infinitesimal constructions are discussed in \S \ref{Infinitesimal.Constructions.Section}. Section  \ref{Examples.Section}
 is devoted to examples. We treat the case of toric varieties, and  describe rather completely the Okounkov body of any big divisor on  a smooth complex surface. Section \ref{Open.Problem.Section} presents some open problems and questions. Finally, we prove in the Appendix a useful technical result concerning intersections of semigroups and cones with linear subspaces.

We are grateful to A. Barvinok, S.-Y. Jow,  A. Khovanskii, D. Maclagan and Ivan Petrakiev for valuable discussions and suggestions.

\setcounter{section}{-1}

\section{Notation and Conventions}
\label{Notation.Conventions.Section}

\noi (0.1)  \ We denote by $\NN$ the additive semigroup of non-negative integers. A \textit{convex body} is a compact convex set $K \subseteq \RR^d$ with non-empty interior.

\noi (0.2) \  We work over an uncountable algebraically closed field $\KK$ of arbitrary characteristic.\footnote{The uncountability hypothesis comes in only on a few occasions -- notably in  \S \ref{Infinitesimal.Constructions.Section}
 and (via Remark \ref{Very.General.Flags.OK}) \S \ref{Multi-graded.Lin.Series.Subsection} -- when we make arguments involving very general position.}
 A \textit{variety} is reduced and irreducible. $\PP(V)$ denotes the projective space of one-dimensional \textit{quotients} of a vector space or vector bundle $V$. The  projective space of one-dimensional subspaces is $\PPP(V)$. A property holds for a \textit{very general} choice of data if it is satisfied away from a countable union of proper closed subvarieties of the relevant parameter space.
 
 \noi (0.3)  \  Let $X$ be a projective variety of dimension $d$. We generally follow the conventions of \cite{PAG} concerning divisors and linear series. Thus  a \textit{divisor} on $X$  always means a Cartier divisor. A divisor $D$ on $X$ is \textit{big} if $\hh{0}{X}{\OO_X(mD)}$ grows like $m^d$. This is equivalent to asking that for any ample divisor $A$ on $X$, $mD - A$ is linearly equivalent to an effective divisor for $m \gg 0$. Bigness makes sense for $\QQ$- and $\RR$-divisors, and the bigness of a divisor depends only on its numerical equivalence class. We refer to \cite[Chapters 2.2.A, 2.2.B]{PAG} for a detailed account.
 
\noi (0.4)  \ We denote by $N^1(X)$ the N\'eron--Severi group of numerical equivalence classes of divisors on a projective variety $X$: it is a free abelian group of finite rank. The corresponding finite-dimensonal $\QQ$- and $\RR$-vector spaces are $N^1(X)_{\QQ}$ and $N^1(X)_\RR$. Inside  $N^1(X)_{\RR}$ one has the pseudo-effective and nef cones of $X$:
 \[  N^1(X)_{\RR} \ \supseteq \ \overline{\tn{Eff}}(X)  \ \supseteq \ \tn{Nef}(X) . \]
 By definition, the pseudo-effective cone $\overline{\tn{Eff}}(X)$ is the closed convex cone spanned by the classes of all effective divisors, whereas 
 \[ \tn{Nef}(X) \ = \ \big \{  \xi \mid  (\xi \cdot C) \ge 0 \text{ for all irreducible cuves $C \subseteq X$} \, \big \}. \]
These are closed convex cones, and the basic fact is: 
 \[
\interior \big( \tn{Nef}(X)  \big ) \ = \ \tn{Amp} (X)  \ \  , \ \ \interior \big( \overline{\tn{Eff}}(X) \big) \ = \ \tn{Big}(X). 
\] 
Here $\tn{Amp} (X) , \tn{Big}(X)   \subseteq N^1(X)_{\RR}$ denote the open cones of ample and big classes respectively. We refer to \cite[Chapters 1.4.C, 2.2.B]{PAG} for details.\footnote{This reference works with complex varieties, but   the discussion of these matters is characteristic-free once one knows that $N^1(X)_{\RR}$ is finite-dimensional.} For a survey of asymptotic invariants of linear series see \cite{ELMNP2}.

\noi (0.5)  \ We recall some facts about semigroups and the cones they span. Let $\Gamma \subseteq \NN^k$ be a finitely generated semigroup, and 
denote by
\[
\Sigma \ = \ \Sigma(\Gamma) \ \subseteq \ \RR^k 
\]
the closed convex cone it spans, i.e. the intersection of all the closed convex cones containing $\Gamma$. Thus $\Sigma$ is a rational polyhedral cone.
Then first of all, $\Sigma \cap \NN^k$ is the saturation of $\Gamma$, so that given
an integer vector $\sigma \in \Sigma \cap \NN^k$, there is a natural number $m = m_\sigma>0$ such that $m_\sigma \cdot \sigma \in \Gamma$. Secondly, if
$\Sigma^\pr \subseteq \RR^k$ is any rational polyhedral cone, then
\[
\CCC\big ( \Sigma^\pr \cap \ZZ^k \big) \ = \ \Sigma^\pr.
\]
See \cite[Proposition 1.1]{Oda}.  Finally, following \cite{Okounkov03}, we will also use some results of Khovanskii \cite{Khovanskii}. Specifically,  assume that $\Gamma$ generates $\ZZ^k$ as a group. Then Proposition 3 of \S 3 of \cite{Khovanskii} asserts  that  there exists an element $z \in \Sigma$ such that any integer vector lying in the translated cone $ z + \Sigma $ actually lies in $\Gamma$, i.e.
\begin{equation} \label{Khovanskii.Equation}
\big( z + \Sigma \big) \ \cap \ \ZZ^k \ \subseteq \ \Gamma. \tag{*}
\end{equation}
Note that the same statement then  holds automatically when $z$ is replaced by $z + z^\pr$ for any $z^\pr \in \Sigma$ (since $z + z^\pr + \Sigma \subseteq z + \Sigma$). So one can assume for instance that $z \in \Gamma$. Observe also that (*) fails if $\Gamma$ does not generate $\ZZ^k$ as a group.


\section{Okounkov's Construction} \label{Okounkov.Construction.Section}

This section is devoted to defining  the Okounkov bodies attached to divisors or graded linear series on an algebraic variety.

Let $X$ be an irreducible variety of dimension $d$. We fix throughout this section a flag 
\begin{equation} \label{admissible.flag.eqn}
Y_{\bull} \ \ : \ \ X \ = \ Y_0 \ \supseteq \ Y_1\ \supseteq \ Y_2 \ \supseteq \ \ldots \ \supseteq \ Y_{d-1} \  \supseteq \  Y_{d} \  = \  \{ \text{pt} \},\end{equation}
of irreducible subvarieties of $X$, where
\[ \text{codim}_X(Y_i) \ = \ i, \] and each $Y_i$  is non-singular at the point $Y_d$. We call this an \textit{admissible flag}. Given $Y\bull$, Okounkov's construction associates a convex body  $\Delta \subseteq \RR^d$ to    a divisor $D$ on $X$ (when $X$ is  complete), or more generally to a graded linear series $W\bull$ on $X$ (without any compactness hypotheses). One proceeds in two steps. First, one uses $Y\bull$ to define a valuative-like function on the sections of any line bundle. Then $\Delta$ is   built from the valuation vectors of all powers of the linear series in question.

\subsection{The Valuation Attached to a Flag} \label{Val.Attached.to.Flag}

Consider any divisor $D$ on $X$. We begin by defining a function
\begin{equation} \label{val.on.sections.eqn}
  \nu \ = \nu_{Y\bull} \ = \nu_{Y\bulll{D}} \ : \ \HH{0}{X}{\OO_X(D)} \lra \ZZ^d \cup \{ \infty\}\ \ , \ \ s \mapsto \nu(s) \, = \, \big( \nu_1(s), \ldots, \nu_d(s)\big) 
\end{equation}
satisfying three valuation-like properties:\footnote{Since we prefer to view $\nu_{Y\bull}$ as being defined on the spaces of sections of different line bundles, it is not strictly speaking a valuation. However for ease of discussion, we will use the term nonetheless.}
\begin{enumerate}
\item [(i).] $\nu_{Y\bull}(s) = \infty$ if and only if $s = 0$;
\item [(ii).] Ordering $\ZZ^d$ lexicographically, one has \[ \nu_{Y\bull}( s_1 + s_2) \, \ge \, \min\big\{ \nu_{Y\bull}(s_1) ,  \nu_{Y\bull}(s_2)\big\}\]  for any non-zero sections $s_1, s_2 \in \HH{0}{X}{\OO_X(D)}$;
\item[(iii).] Given non-zero sections  $s \in \HH{0}{X}{\OO_X(D)}$ and $ t \in \HH{0}{X}{\OO_X(E)}$, 
\[  \nu_{Y\bulll{D+E}}(s \otimes t) \ = \ \nu_{Y\bulll{D}}(s) \, + \, \nu_{Y\bulll{E}}(t).
\]
\end{enumerate}
In a word, the plan is produce the integers   $\nu_i(s)$ inductively by restricting to each subvariety  in the flag, and considering the order of vanishing along the  next smallest. 

Specifically, we may suppose after replacing $X$ by an open set that each $Y_{i+1}$ is a Cartier divisor on  $Y_{i}$: for instance one could take all the $Y_i$ to be smooth. Given \[ 0 \ne s \in \HH{0}{X}{\OO_X(D)},\]
set to begin with
\[ \nu_1\,  = \, \nu_1(s) \ = \ \ord_{Y_1}(s). \]
After choosing a local equation for $Y_1$ in $X$, $s$ determines a section
\[  \tilde s_1 \ \in \ \HH{0}{X}{\OO_X(D-\nu_1Y_1)} \]
that does not vanish (identically) along $Y_1$, and so we get by restricting a non-zero section
\[  s_1 \in \HH{0}{Y_1}{\OO_{Y_1}(D-\nu_1Y_1)}. \]
Then take 
\[
\nu_2 \, = \, \nu_2(s) \ = \ \ord_{Y_2}(s_1). 
\] 
In general, given integers $a_1, \ldots, a_i \ge 0$ denote by $\OO(D-a_1Y_1 - a_2Y_2 - \ldots - a_i Y_i)_{|Y_i}$
the line bundle 
\[  \OO_X(D)_{|Y_i} \otimes \OO_X(-a_1Y_1)_{| Y_i} \otimes \OO_{Y_1}(-a_2 Y_2)_{|Y_i} \otimes \ldots \ \otimes \OO_{Y_{i-1}}(-a_i Y_i)_{|Y_i} \]
on $Y_i$. Suppose inductively that for $i \le k$ one has constructed non-vanishing sections
\[ s_i \ \in \ \HHH{0}{Y_i}{\OO(D-\nu_1Y_1 - \nu_2Y_2 - \ldots - \nu_i Y_i)_{|Y_i} } ,  
\]  
with  $\nu_{i+1}(s) = \ord_{Y_{i+1}}(s_i)$, so that in particular 
\[ \nu_{k+1}(s) 
\ = \ \ord_{Y_{k+1}}(s_k). \]
 Dividing by the appropriate power of a local equation of $Y_{k+1}$ in $Y_k$ yields a section 
\[
 \tilde s_{k+1} \, \in \, \HHH{0}{Y_{k}}{\OO \big ( D-\nu_1Y_1 - \nu_2Y_2 - \ldots - \nu_k Y_k\big )_{|Y_k}\otimes \OO_{Y_k}(-\nu_{k+1}Y_{k+1})} \]
not vanishing along $Y_{k+1}$. Then take
\[  s_{k+1} \ = \ \tilde s_{k+1} | Y_{k+1} \  \in \ 
\HHH{0}{Y_{k+1}}{\OO \big(D-\nu_1Y_1 - \nu_2Y_2 - \ldots - \nu_{k+1}Y_{k+1}\big)_{|Y_{k+1}}} \]
to continue the process. Note that while the sections $\tilde s_i$ and $s_i$ will depend on the choice of a local equation of each $Y_i$ in $Y_{i-1}$, the values $\nu_i(s) \in \NN$ do not. 
It is immediate that properties (i) -- (iii) are satisfied.

\begin{example} \label{Lex.valuation.Pd}
On  $X = \mathbf{\PP}^d$, let $Y\bull$ be the flag of linear spaces defined in homogeneous coordinates $T_0, \ldots, T_d$ by
 $Y_i = \{T_1 = \ldots = T_{i} = 0\}$ and take  $\linser{D}$ to be the linear system of hypersurfaces of degree $m$. Then    $\nu_{Y\bull}$ is the lexicographic valuation determined on monomials of degree $m$ by
 \[   \nu_{Y\bull} (T_0^{a_0} T_1^{a_1} \cdots T_d^{a_d}) \ = \ (a_1 , \ldots , a_{d}). \qed \]
\end{example}

\sbl
\begin{example} \label{Van.Seq.Curves.Ex}
Let $C$ be a smooth projective curve of genus $g$, and fix a point $P \in C$, yielding the flag $C \supseteq \{P\}$. Given a divisor $D$ on $C$, the image of the resulting map
\[
\nu \, : \, \HH{0}{C}{\OO_X(D)} - \{ 0 \} \lra \ZZ 
\]
is the classical vanishing sequence of the complete linear series $\linser{D}$ at $P$ (cf \cite[p. 256]{Harris-Morrison}). If $c= \deg (D) \ge 2g +1$ this consists of $c + 1 - g$ non-negative integers lying in the interval $[0,c]$.  If $\tn{char}\, 
\KK = 0$ then for most choices of $P$ one has
\[
\Image (\nu)  \ = \ \big \{  0, 1, \ldots, c-g \big \},
\]
but for special $P$ there will  be  gaps. \qed
\end{example}

The following lemma expresses a basic property of the valuation $\nu_{Y\bull}$ attatched to a flag $Y\bull$.
\begin{lemma} \label{No.of.Val.Vectors}
Let $W \subseteq \HH{0}{X}{\OO_X(D)}$ be a subspace. Fix
\[  a \ = \ (a_1, \ldots, a_d) \ \in \ \ZZ^d \]
and set
\[   W_{\ge a} \ = \ \big \{ s \in W \, \big | \, \nu_{Y\bull}(s) \ge a \ \big \} \ \ ,  \ \ W_{> a} \ = \ \big \{ s \in W \, \big | \, \nu_{Y\bull}(s) > a \ \big \} , \]
where as above $\ZZ^d$ is ordered lexicographically. Then  
\[
\dim \big ( W_{\ge a}\, / \, W_{> a} \big ) \ \le \ 1. 
\]In particular, if $W$ is finite dimensional  then the number of valuation vectors arising from sections in $W$ is equal to the dimension of $W$: 
\[
\# \Big(  \ \Image \big( \,  (W - \{ 0 \}) \overset{\nu}\lra \ZZ^d\,  \big) \ \Big ) \ = \ \dim W.
\]
\end{lemma}

\begin{proof}
In fact, it is a consequence of the definition that $W_{\ge a}\, / \, W_{> a}$ injects into the space of sections of the one-dimensional skyscraper sheaf
\[
\OO \big(D -a_1Y_1 - \ldots - a_{d-1}Y_{d-1}\,  \big)_{|Y_{d-1}} \otimes \frac{ \OO_{Y_{d-1}}(-a_d Y_{d})}{ \OO_{Y_{d-1}}(-(a_d+1) Y_{d})}
\]
on the curve $Y_{d-1}$. The second statement follows. 
\end{proof}

We conclude this subsection with two technical remarks that will be useful later.
\begin{remark} \textbf{(Partial flags.)} 
\label{Partial.Flag.Remark}
A similar construction is possible starting from a partial flag
\[
Y_{\bull} ^\pr \ \ : \ \ X \ = \ Y_0 \ \supseteq \ Y_1\ \supseteq \ Y_2 \ \supseteq \ \ldots \ \supseteq \ Y_{r-1} \  \supseteq \  Y_{r}  \] 
where $\codim_X Y_i = i$ and each $Y_i$ is non-singular at a general point of $Y_r$. In fact,  just as above, such a flag defines for every $D$ a map
\begin{equation} 
  \nu_{Y\bull^\pr}   : \ \HH{0}{X}{\OO_X(D)} \lra \ZZ^r \cup \{ \infty\} \end{equation}
  satisfying the analogues of (i) - (iii). When $Y^\pr\bull$ is the truncation of a full flag $Y\bull$, it is natural to consider the value group $\ZZ^r$ of $\nu_{Y\bull^\pr}$ as a subgroup of the value group $\ZZ^d$ of $\nu_{Y\bull}$ via inclusion on the first $r$ coordinates.  Observe however that the analogue of the previous lemma fails for the valuation defined by an incomplete flag. \qed
\end{remark}

\begin{remark} \textbf{(Sheafification.)} \label{sheafification.remark}
Of course one can work with a line bundle $L$ on $X$ in place of a divisor, and with this notation it  is worthwhile  to note that the construction of $\nu_{Y\bull}$ sheafifies. Specifically, fix $\sigma = (\sigma_1, \ldots, \sigma_d)  \in \ZZ^d$. Then given any line bundle $L$ on $X$, there exists a coherent subsheaf $L^{\ge \sigma} \subseteq L$ characterized by  the property that \[
L^{\ge \sigma}(U) \ = \ \big \{ \,  s \in L(U) \, | \,  \nu_{Y\bull | U}(s) \ge \sigma \, \big \} 
\]
for any open set $U \subseteq X$, where $Y\bull | U$ is the (possibly partial) flag obtained by restricting $Y\bull$ to $U$, and  where as above $\ZZ^d$ is ordered lexicographically. Supposing first that  $Y_{i+1}$ is a Cartier divisor in $Y_i$ for every $i$, one can construct $L^{\ge \sigma}$ by an iterative procedure. In fact, take to begin with  $L^{\ge (\sigma_1) } =L(-\sigma_1 Y_1)$. Then define $L^{\ge (\sigma_1, \sigma_2)} $ to be the inverse image of the subsheaf $L (-\sigma_1Y_1 - \sigma_2 Y_2)_{|Y_1} \subseteq L  (-\sigma_1 Y_1)_{|Y_1}$ under the surjection $L(-\sigma_1 Y_1) \lra L(-\sigma_1Y_1)_{|Y_1} $:
\[
\xymatrix{
L^{\ge (\sigma_1, \sigma_2)}\ar@{->>} [r] 
\ar@{^{(}->} [d]&L(-\sigma_1Y_1 - \sigma_2 Y_2)_{|Y_1}  \ar@{^{(}->} [d]\\
L^{\ge (\sigma_1)} \ = \ L(-\sigma_1 Y_1) \ar@{->>} [r] &  L(-\sigma_1Y_1)_{|Y_1} .}
\]
Then continue in this manner to define inductively $L^{\ge (\sigma_1, \ldots, \sigma_k)}$ when each $Y_{i+1} \subseteq Y_i$ is a divisor. In general, take an open neighborhood $j :  V \subseteq X$ of $Y_d$, and put
\begin{equation} \label{Close.Up.Sheaf.Eqn}
L^{\ge \sigma} \ = \ j_* \big( (L|V)^{\ge \sigma} \big ) \, \cap \, L,
\end{equation}
the intersection taking place  in the constant sheaf $L \otimes \KK(X)$ determined by the stalk of $L$ at the generic point of $X$. Observe that a similar construction is possible starting with a partial flag. \qed
\end{remark}

\subsection{Construction of the Okounkov Body }  
Consider as above a divisor $D$  on $X$. We assume in this subsection that $X$ is projective, so that in particular the spaces of sections $\HH{0}{X}{\OO_X(mD)}$ are finite-dimensional.  It follows from the valuative properties of $\nu_{Y\bull}$ that the valuation vectors of sections of $\OO_X(D)$ and its powers form an additive semigroup in $\NN^{d}$. 
However, as in \cite{Okounkov03}  it will be convenient to work with a variant that keeps track of the grading:
\begin{definition} \textbf{(Graded semigroup of a divisor).} \label{Graded.Semigroup.of.Divisor}
The 
\textit{graded semigroup} of   $D$ is the sub-semigroup
\[ \Gamma  (D) \ = \ \Gamma_{Y\bull}(D)  \ = \ \big\{ ( \nu_{Y\bull}(s), m)  \, \big | \, 0 \ne s \in \HH{0}{X}{\OO_X(mD)} \, , \,  m \ge 0 \,  \big \}   \]
of  $\NN^d \times \NN = \NN^{d+1}$. \qed
 \end{definition}
\noi We consider $\Gamma(D)$ also as a subset
 \[  \Gamma(D) \ \subseteq \ \ZZ^{d+1} \ \subseteq \ \RR^{d+1} \]
 via the standard inclusions   $\NN \subseteq \ZZ \subseteq \RR$. 

\sbl

\begin{example} \textbf{(Failure of finite generation).} \label{Fail.Finite.Generation} The graded semigroup $\Gamma(D)$ is typically not finitely generated, even in very  simple situations.  For instance, consider as in  Example \ref{Van.Seq.Curves.Ex} a divisor $D$  of degree $c \ge 2g +1$ on a smooth complex curve $C$ of genus $g$. Working with the flag $C \supseteq \{ P \}$ for a very general choice of $P\in C$,  the semigroup in question is given by
\[ \Gamma(D) \ = \ \big \{ \, (0, 0) \, \big \} \ \cup \ \big \{ \, (k, m) \,  \big  | \, m \ge 1 \ , \ 0 \le k \le mc - g  \, \}  \ \subseteq  \ \NN^2. \] 
But as soon as  $g \ge 1$ this fails to be finitely generated. \qed
\end{example}

\sbl

  Writing $\Gamma = \Gamma(D)$, denote by 
\[ \Sigma(\Gamma)\  \subseteq \ \RR^{d+1} \] 
 the closed convex cone (with vertex at the origin) spanned by $\Gamma$, i.e. the intersection of all the closed convex cones containing $\Gamma$. 
 The Okounkov body of $D$ is then the base of this cone:
 \begin{definition} \textbf{(Okounkov body).}\label{Def.Ok.Body}
 The \textit{Okounkov body} of $D$ (with respect to the fixed flag $Y\bull$) is the compact convex set
 \[  \Delta(D) \ = \ \Delta_{Y\bull}(D) \ = \ \Sigma(\Gamma) \, \cap \,  \big ( \RR^d  \times \{1\}\big ). \]
 We view $\Delta(D)$ in the natural way as a closed convex subset of $\RR^d$; compactness follows from Lemma \ref{boundedness} below, which shows that it is bounded. Occasionally, when we want to emphasize the underlying variety $X$, we write $\Delta_{Y\bull}(X; D)$.\qed
 \end{definition}
 \noi Alternatively, let 
 \[  \Gamma(D)_m  \ = \ \Image \Big (  \big( \,\HH{0}{X}{\OO_X(mD)} - \{ 0 \} \, \big) \overset{\nu} \lra \ZZ^d\Big). \]
 Then 
 \begin{equation} \label{Alt.Descr.Okounkov.Body}
 \Delta(D) \ = \ \textnormal{ closed convex hull } \Big( \bigcup_{m \ge 1} \, \tfrac{1}{m} \cdot \Gamma(D)_m \Big) \ \subseteq \ \RR^d. 
 \end{equation}
 Observe that by construction $\Delta(D)$ lies in the non-negative orthant of $\RR^d$. 
 
 \begin{remark} \textbf{(Line bundles).}
 Sometimes it will be preferable to use the language of line bundles. If $L$ is a line bundle on $X$, we write $\Gamma(L) \subseteq \ZZ^{d+1}$ and $\Delta(L) \subseteq \RR^d$ for the graded semigroup and Okounkov body of a divisor $D$ with $\OO_X(D) = L$. \qed
 \end{remark}
 \sbl
  
   The compactness of $\Delta(D)$ follows from
 \begin{lemma} \textnormal{\bf(Boundedness).} \label{boundedness}
 The Okounkov body $\Delta(D)$ lies  in a  bounded subset of $\RR^d$.
 \end{lemma}
 \begin{proof}
It suffices to show that if  $b  \gg 0$ is a sufficiently large integer (depending on $D$ as well as $Y\bull$), then
\begin{equation} \label{valuation.bounded.eqn}
  \nu_i(s) \ <\ mb \ \ \text{for every }  1 \le i \le d \ , m > 0 \ ,  \text{ and  } 0 \ne s \in \HH{0}{X}{\OO_X(mD)}. \end{equation}
To this end fix an ample divisor $H$, and choose first of all an integer $b_1$ which is sufficiently large that 
\begin{equation} \big( D - b_1 Y_1 \big) \cdot H^{d-1} < 0. \notag \end{equation}
This guarantees that $\nu_1(s) \le m b_1 $ for all $s$ as above. Next, choose $b_2 $ large enough so that on $Y_1$ one has 
\begin{equation} \big( (D - aY_1)_{|Y_1} - b_2Y_2 \big) \cdot H^{d-2} \ < \ 0 \notag \end{equation}
for all real numbers $0 \le a \le b_1$. Then $\nu_2(s) \le mb_2$ for all $0 \ne s \in \HH{0}{X}{\OO_X(mD)}$. 
Continuing in this manner one constructs integers $b_i > 0$ for $i = 1, \ldots, d$ such that $\nu_i(s) \le mb_i$, and then it is enough to take $b = \max \{ b_i \}$. \end{proof}

\begin{remark}{ \bf{(Extension to several divisors).}} \label{Extn.of.Bddness.Lemma.Remark}
Observe for later reference that a similar argument proves an analogous statement  for several divisors. Specifically, fix divisors $D_1, \ldots, D_r$ on $X$. We assert that then there exists a constant $b \gg 0$ with the property that 
\[ \nu_i(s) \ \le \ b \cdot \sum | m_i| \]
for any integers $m_1, \ldots, m_r$ and any non-zero section $0 \ne s \in \HH{0}{X}{\OO_X(m_1D_1 + \ldots + m_rD_r)}$. In fact, first choose $b_1 > 0$ such that
\[  \left ( \sum \lambda_i D_i \, - \, b_1Y_1\right ) \cdot H^{d-1} \ < \ 0 \] whenever $\sum  | \lambda_i | \le 1$. This implies that $\nu_1(s) < b_1 \cdot \sum |m_i| $. Next fix $b_2 > 0$ so that 
\[  \left( \big( \sum \lambda_i D_i - aY_1\big)|_{ Y_1} - b_2 Y_2 \right) \cdot H^{d-1} \ < \ 0 \]
for $\sum |\lambda_i | \le 1$ and $0 \le a \le b_1$. This yields $\nu_2(s) < b_2 \cdot \sum |m_i|$, and as above one continues in this manner. \qed
\end{remark}
 
 \begin{remark}
 For arbitrary divisors  $D$ it can happen that $\Delta(D) \subseteq \RR^d$ has empty interior, in which event $\Delta(D)$ isn't  actually  a convex body. (For instance, for the zero divisor $D = 0$, $\Delta(D)$ consists of a single point.) However we will be almost exclusively interested in the case when $D$ is big, and then 
 $ \text{int} \big ( \Delta(D) \big)$ is indeed non-empty. \qed
 \end{remark}

 \sbl
 \begin{example} \textbf{(Curves.)}
 \label{OB.Curves.Ex}
  Let $D$ be a divisor of degree $c > 0$ on a smooth curve $C$ of genus $g$. Then it follows from Example  \ref{Van.Seq.Curves.Ex}    that 
 \[  
 \Delta (D) \ = \ [0,c] \ \subset \ \RR  \]
 is the closed interval of length $c$ for any choice of flag $C \supseteq \{ P \} $.  \qed
   \end{example}


\sbl
\begin{example}
Let $X = \PP^d$, $D = H$  a hyperplane divisor, and take $Y\bull$ to be the linear flag appearing in Example \ref{Lex.valuation.Pd}. Then it follows immediately from that example that $\Delta(D)$ is the simplex
\[  \big \{  (\xi_1, \ldots , \xi_d) \in \RR^d \, \big | \, \xi_1 \ge 0 \, , \,  \ldots \, , \,  \xi_d \ge 0  \  ,  \ \sum \xi_i \le 1 \, \big \}. \] 
This is a special case of Proposition \ref{Toric.Example.Prop}, which computes the Okounkov body of a toric divisor on a toric variety with respect to a toric flag. \qed
\end{example}

\subsection{Graded Linear Series.} We assumed in the previous paragraphs that $X$ is projective in order that the spaces $
\HH{0}{X}{\OO_X(mD)}$ appearing there be finite dimensional and so that the boundedness statement \ref{boundedness} holds.  However one can also define Okounkov bodies in a more general setting that does not require any completeness hypotheses.

Specifically, let $D$ be a divisor on $X$, which is no longer assumed to be complete, and let $ W\bull = \{ W_k \}$ be a graded linear series on $X$ associated to $D$. Recall that this consists of finite dimensional subspaces
\[  W_k \ \subseteq \ \HH{0}{X}{\OO_X(kD)}  \]
for each $k \ge 0$, with $W_0 = \KK$,  which are required to satisfy
the inclusion \begin{equation} W_k \cdot W_\ell \ \subseteq \ W_{k + \ell}   \tag{*} \end{equation}
for all $k , \ell \ge 0 $.  Here the product on the left denotes the image of $W_k \otimes W_\ell$ under the multiplication map $\HH{0}{X}{\OO_X(kD)} \otimes \HH{0}{X}{\OO_X(\ell D)} \lra \HH{0}{X}{\OO_X(({k + \ell})D)}$.  Equivalently, (*) demands that $R(W\bull) = \oplus W_m$ be a graded subalgebra of the section ring $R(X,D) = \oplus \HH{0}{X}{\OO_X(mD)}$. We refer to \cite[Chapter 2.4]{PAG} for further discussion and examples. 

One now proceeds exactly as before:
\begin{definition} Let $W\bull$ be a graded linear series on $X$ belonging to a divisor $D$. The \textit{graded semigroup} of $W\bull$ is
\[ \Gamma(W\bull) \ = \ \Gamma_{Y\bull}(W\bull) \ = \ \  \big\{ ( \nu_{Y\bull}(s), m)  \, \big | \, 0 \ne s \in W_m \, , \,  m \ge 0 \,  \big \}  \ \subseteq \ \ZZ^{d+1} .\]
The \textit{Okounkov body} of $W\bull$ is the base of the closed convex cone spanned by $\Gamma = \Gamma(W\bull)$:
 \[  \Delta(W\bull) \ = \ \Delta_{Y\bull}(W\bull) \ = \ \CCC(\Gamma) \, \cap \,  \big ( \RR^d  \times \{1\}\big ). \qed \]
\end{definition}
\noi Again $\Gamma(W\bull)$ is a closed convex subset of $\RR^d$. The alternative description \eqref{Alt.Descr.Okounkov.Body} also extends to the present context, namely:
 \begin{equation} \label{Alt.Descr.Okounkov.Body.GLS}
 \Delta(W\bull) \ = \ \textnormal{ closed convex hull } \Big( \bigcup_{m \ge 1} \, \tfrac{1}{m} \cdot \Gamma(W\bull)_m  \Big) \ \subseteq \ \RR^d,
 \end{equation}
where $\Gamma(W\bull)_m  \subseteq \NN^d$ denotes the image of $\nu_{Y\bull}: (W_m - \{0\}) \lra \ZZ^d$.

\begin{remark} {\textbf{(Pathology).}}
When $X$ is not complete,   $\Delta(W\bull)$ may fail to be a bounded subset  of $\RR^d$. (For example let $X = \mathbf{A}^1$, $D = 0$, and take $W_m$ to be the set of all polynomials of degree $\le m^2$.) In the sequel we will always impose further conditions to rule out this sort of pathology. \qed
\end{remark}

We conclude this section  by observing that essentially every convex body arises as the Okounkov body of a graded linear series on a projective variety.
\begin{proposition} \label{Realize.Convex.Body.Prop}
Let $K \subseteq \RR^d$ be an arbitrary convex body. Then after possibly translating and scaling $K$, there exists a graded linear series $W\bull$ on $\PP^d$ associated to the hyperplane divisor, and a flag $Y\bull$ on $ \PP^d$, such that $K = \Delta(W\bull)$.
\end{proposition}

\begin{proof}
We mimic a construction used by Wolfe \cite{Wolfe}, which was in turn inspired by \cite{Mustata}.  Specifically, let $T \subseteq \RR^d$ be the simplex 
\[  T \ = \ \big \{  (\xi_1, \ldots , \xi_d) \in \RR^d \, \big | \, \xi_1 \ge 0 \, , \,  \ldots \, , \,  \xi_d \ge 0  \  ,  \ \sum \xi_i \le 1 \, \big \}. \] 
We may assume that $K \subseteq T$. Given $v \in mT \cap \ZZ^d$, we view $v$ as the exponent vector of a monomial $x^v$ of degree $\le m$ in variables $x_1, \ldots, x_d$. Denote by $W_m^\pr$ the $\KK$-linear span of the monomials corresponding to integer points in $mK$:
\[  W^\pr_m \ = \ \text{span}_\KK  \ \big \langle x^v \ \big | \ v \in mK \cap \ZZ^d \ \big \rangle. \]
Then evidently $W^\pr_m \cdot W^\pr_\ell \subseteq W^\pr_{m + \ell}$ for all $m , \ell 
\ge 0$. On the other hand, $W^\pr_m$ determines by homogenization a subspace
\[ W_m \  \subseteq \ \HH{0}{\PP^d}{\OO_{\PP^d}(m)}, \]
and these form a graded linear series $W\bull$. If $Y\bull$ is the linear flag appearing in Example \ref{Lex.valuation.Pd}, then
\[ \Gamma(W\bull)_m  \ = \ mK \, \cap \, \ZZ^d, \]
where as above $ \Gamma(W\bull)_m  = \Image \big( (W_m - \{ 0 \}) \lra \ZZ^d\big)$. Therefore
\[
  \bigcup_{m \ge 1} \, \tfrac{1}{m} \cdot \Gamma(W\bull)_m \ = \ K \, \cap \, \QQ^d. 
\]
But $ K  = \text{closure} \big( K \cap \QQ^d \big)$ since $K$ is the closure of its interior. So it follows from \eqref{Alt.Descr.Okounkov.Body.GLS} that $K = \Delta(W\bull)$, as required.
\end{proof}

\section{Volumes of Okounkov Bodies } \label{Volume.Ok.Body.Section}

In this section we establish the basic Theorem \ref{Vol.Okounkov.Body.Eqn} computing the volume of $\Delta(D)$, and we introduce some conditions leading to the corresponding statement for graded linear series. In the final subsection we discuss  restricted linear series.

\subsection{Semigroups.} 
Following Okounkov \cite{Okounkov03}, the plan is to deduce the theorem in question from some results of Khovanskii \cite{Khovanskii} on sub-semigroups of $\NN^{d+1}$. 
Given any semigroup $\Gamma \subseteq \NN^{d+1}$, set
\begin{equation} \label{K.Delta.Defn}
\begin{aligned}
\Sigma \ &= \ \Sigma(\Gamma) \ = \ \CCC(\Gamma) \ \subseteq \RR^{d+1} , \\
\Delta \ &= \ \Delta(\Gamma) \ = \ \Sigma \, \cap \, \big( \RR^d \times \{ 1 \} \big). \end{aligned}
\end{equation}
Moreover for $m \in \NN$, put
\begin{equation} \label{Def.Gamma.m}
\Gamma_m \ = \ \Gamma \, \cap \,  \big( \NN^d \times \{ m \} \big),
 \end{equation}
which we view  as a subset $\NN^d$. 
We do not assume that $\Gamma$ is finitely generated, but we will suppose that it satisfies three conditions:
\begin{gather}
\label{zero} \Gamma_0 \ = \ \{0\} \, \in \, \NN^d; \\
 \label{bounded}
\begin{gathered} \text{$\exists$ finitely many vectors $(v_i, 1)$ spanning a semi-group $B \subseteq \NN^{d+1}$   such that}\\ \Gamma \subseteq B;   
 \end{gathered}\\
  \label{full} \Gamma \text{ generates } \ZZ^{d+1} \text{ as a group. } \end{gather}
 Observe that these conditions imply that $\Delta(\Gamma)$ -- which we consider in the natural way as a subset of $\RR^d$ --  is a convex body.

The essential point is the following
\begin{proposition}  \label{semigroup.proposition}
Assume that $\Gamma$ satisfies \eqref{zero} -- \eqref{full}. Then 
\[ 
\lim_{m \to \infty} \, \frac{\#   \, \Gamma_m }{m^d} \ = \ \vol_{\RR^d}( \Delta) . 
\]
\end{proposition}
\noi Here $\vol_{\RR^d}$ denotes the standard Euclidean volume on $\RR^d = \ZZ^d \otimes \RR$, normalized so that the unit cube $[0,1]^d$  has volume $= 1$.

\begin{proof} We repeat Okounkov's argument from \cite[\S 3]{Okounkov03}. 
One has \[  \Gamma_m \ \subseteq \ m \Delta \, \cap \, \ZZ^d, \]
and since
\[
\lim_{m \to \infty} \, \frac { \# \big( m \Delta \, \cap \, \ZZ^d \big)}{m^d} \ = \ \vol_{\RR^d}(\Delta)
\]
it follows that
\begin{equation}
\limsup_{m \to \infty} \, \frac{\#   \, \Gamma_m }{m^d} \le \ \vol_{\RR^d}( \Delta) . 
\notag \end{equation}
For the reverse inequality, assume to begin with that $\Gamma$ is finitely generated. Khovanskii \cite[\S 3, Proposition 3]{Khovanskii} shows that in this case there exists a vector $\gamma \in \Gamma$ such that
\[  \big ( \Sigma + \gamma \big) \, \cap \, \NN^{d+1} \ \subseteq \ \Gamma: \]
here one uses that $\Gamma$ generates $\ZZ^{d+1}$ as a group (see (0.5) in \S \ref{Notation.Conventions.Section}).
But evidently
\[
\lim_{m \to \infty} \frac {\# \left \{ (\Sigma + \gamma) \, \cap \, (\NN^d \times \{m\} ) \right \}}{m^d} \ = \ \vol_{\RR^d}(\Delta),
\]
and hence
\begin{equation}
\liminf_{m \to \infty} \, \frac{\#   \, \Gamma_m }{m^d} \ge \ \vol_{\RR^d}( \Delta) . 
\tag{*} \end{equation}
This proves the theorem when $\Gamma$ is finitely generated. 

In general, choose finitely generated sub-semigroups
\[  \Gamma^1 \ \subseteq \ \Gamma^2 \ \subseteq \ \ldots \ \subseteq  \
  \Gamma, \]
each satisfying \eqref{zero}--\eqref{full}, in such a manner that
$\cup \, \Gamma^i = \Gamma$. Then $\# \, \Gamma_m \ge \# \, (\Gamma^i)_m$ for all $m \in \NN$. Writing $\Delta^i = \Delta(\Gamma^i)$, it follows by applying (*) to $\Gamma^i$ that
\[
\liminf_{m \to \infty} \, \frac{\#   \, \Gamma_m }{m^d} \ge \ \vol_{\RR^d}( \Delta^i) .\]
for all $i$.
But $ \vol_{\RR^d}(\Delta^i) \to \vol_{\RR^d}(\Delta) $ and so (*) holds also for $\Gamma$ itself.
\end{proof}

\subsection{Global Linear Series.}
We return now to the geometric setting of Section 1. For global linear series, the discussion just completed applies without further ado thanks to: 
\begin{lemma} \label{Gamma.L.generates.Group}
Let $X$ be a projective variety of dimension $d$, and let $Y\bull$ be any admissible flag of subvarieties of $X$. If $D$ is any big divisor on $X$, then the graded semigroup \[ \Gamma \ = \ \Gamma_{Y\bull}(D)  \ \subseteq  \ \NN^{d+1} \]associated to $D$   satisfies the three conditions \eqref{zero} -- \eqref{full}. 
\end{lemma}

\begin{proof} That $\Gamma_0 = 0$ is clear. As for \eqref{bounded}, we noted in the proof of Lemma \ref{boundedness} that there is an integer $b \gg 0$ with the property that 
\[   \nu_i(s) \ \le \ mb \ \ \text{for every }  1 \le i \le d \ \text{ and every } 0 \ne s \in \HH{0}{X}{\OO_X(mD)}. \]
This implies that $\Gamma$ is contained in the semigroup $B \subseteq \NN^{d+1}$ generated by  all vectors $(a_1, \ldots, a_d, 1) \in \NN^{d+1}$ with $0 \le a_i \le b$. It remains to show that $\Gamma$ generates $\ZZ^{d+1}$ as a group. 

To this end  write $D = A - B$ as the difference of two very ample divisors. By adding a further very ample divisor to both $A$ and $B$, we can suppose that there exist sections $s_0 \in \HH{0}{X}{\OO_X(A)}$ and $ t_i \in \HH{0}{X}{\OO_X(B)}$ for $0 \le i \le d$ such that 
\[ \nu(s_0) \ = \  \nu(t_0) \ = \ 0 \ \ , \ \ \nu_i(t_i) \ = \  e_i \ \  (1 \le i \le d) ,\]
where $e_i \in \ZZ^d$ is the $i^{\text{th}}$ standard basis vector. In fact, it suffices  that $t_i$ is non-zero on $Y_{i-1}$, while the restriction $t_i|{Y_{i-1}}$ vanishes simply along $Y_i$ in a neighborhood of the point $Y_{d}$.  Next, since $D$ is big, there is an integer $m_0 = m_0(D)$ such that $mD - B$ is linearly equivalent to an effective divisor $F_m$ whenever $m \ge m_0$. Thus $mD \lin B + F_m$, and if $f_m \in \ZZ^d$ is the valuation vector of a section defining $F_m$, then we find that
\begin{equation}
(f_m , m)  \ , \ (f_m + e_1, m) \ , \ \ldots \ , \ (f_m + e_d, m)  \ \in \ \Gamma. \tag{*}
\end{equation}
On the other hand, $(m+1)D \lin A + F_m$, and so $\Gamma$ also contains the vector  $(f_m, m+1)$. Combined with  (*), this exhibits the standard basis of $\ZZ^{d+1}$ as lying in the group generated by $\Gamma$. 
\end{proof}

One then gets:
\begin{theorem} \label{Vol.eq.Vol}
Let $D$ be a big divisor on a projective variety $X$ of dimension $d$. Then \[  \vol_{\RR^d} \big ( \Delta(D) \big) \ = \ \frac{1}{d!} \, \vol_X(D), \]
where the Okounkov body $\Delta(D)$ is constructed with respect to any choice of an admissible flag $Y\bull$ as in \eqref{admissible.flag.eqn}. \end{theorem}

\begin{proof}
Let $\Gamma = \Gamma(D)$ be the graded semigroup of $D$ with respect to $Y\bull$. Proposition \ref{semigroup.proposition} applies  thanks to the previous lemma, and  hence
\begin{equation} \label{Volume.Counts.Points.Eqn}
\vol_{\RR^d} \big ( \Delta(D) \big) \ = \ \lim_{m \to\infty}\, \frac{ \# \, \Gamma(D)_m}{m^d}. 
\end{equation}
On the other hand,  it follows from Lemma \ref{No.of.Val.Vectors} that $\# \Gamma(D)_m = \hh{0}{X}{\OO_X(mD)}$, and then by definition the limit on the right in \eqref{Volume.Counts.Points.Eqn}
 computes $\frac{1}{d!}\vol_X(D)$. 
\end{proof}

\subsection{Conditions on Graded Linear Series.} 
\label{Conds.GLS.subsection}
Turning to the setting of graded linear series, suppose that $X$ is an irreducible variety of dimension $d$, and that $W\bull$ is a graded linear series associated to a divisor $D$ on $X$. Fix an admissible flag $Y\bull$. We seek conditions on $W\bull$ and on $Y\bull$ in order that the corresponding graded semigroup $\Gamma_{Y\bull}(W\bull) \subseteq \NN^{d+1}$ satisfies the conditions \eqref{bounded} and \eqref{full}. 

For \eqref{bounded}, we propose the following:
\begin{definition} \textbf{(Condition (A)).} 
We say that $W\bull$ satisfies {condition} (A) with respect to $Y\bull$ if there is an integer $b \gg 0$ such that for every $0 \ne s \in W_{m}$, 
\[  \nu_i(s) \ \le \ mb  \]
for all $1 \le i \le d$. \qed
\end{definition}
\noi As in the proof of Lemma \ref{Gamma.L.generates.Group}, this indeed implies that  \eqref{bounded} holds. We note that (A) holds automatically if $X$ is projective: this was established in the course of proving Lemma \ref{boundedness}.

Concerning the spanning condition \eqref{full}, we start with:
 \begin{definition} \textbf{(Condition (B)).} \label{Cond.B.Def}
We will say that $W\bull$ satisfies  {condition} (B)  if  $W_m \ne0$ for all $m \gg 0$, and if for all sufficiently large $m$ the rational map
\[  \phi_m \, : \, X \dra \PP \ = \ \PP(W_m) \]
defined by $\linser{W_m}$ is birational onto its image. \qed
\end{definition}
\noi Equivalently, one could ask that $W_k \ne 0$ for all sufficiently large $k$, and that $\phi_m$ be birational onto its image for any one  $m > 0$. 

One then has
 \begin{lemma} \label{Cond.F.generates.generic.flag}
 If $W\bull$ satisfies condition \tn{(B)}, then there exists an admissible flag $Y\bull$ on $X$ with respect to which the graded semigroup $\Gamma_{Y\bull}(W\bull) \subseteq \NN^{d+1}$ generates $\ZZ^{d+1}$ as a group. 
 \end{lemma}

\begin{proof} (Compare \cite{Okounkov03}.)
Assume  that  $\linser{W_\ell}$ determines a birational embedding \[ \phi \, = \,  \phi_\ell : X \dra \PP = \PP(W_\ell). \]
Let $y \in X$ be any smooth point at which $\phi_\ell$ is defined and locally an isomorphism onto its image, and which in addition is not contained in the base locus of $\linser{W_q}$ for some fixed large integer $q$ relatively prime to $\ell$.
Take
\[  Y_{\bull} \ \ : \ \ X \ \supseteq \   Y_1\, \supseteq \, Y_2 \, \supseteq \,\ldots \, \supseteq \, Y_{d-1} \,  \supseteq \,  Y_{d} = \{ y \} \]
to be any admissible flag centered at $y$.  Then for any $p \gg 0$      one can find (by pulling back suitable hypersurfaces in $\PP$)
 sections
$t_0, t_1, \ldots, t_d \in W_{p\ell}$ such that
\[  \nu_{Y\bull}(t_0) \, = \, 0 \ \ ,  \ \ \nu_{Y\bull}(t_i) \ = \ e_i    \ \ (1 \le i \le d),\]
where $e_i \in \NN^d$ is the $i^{\text{th}}$ standard basis vector. On the other hand,  there exists $s_0 \in W_q$ such that $\nu_{Y\bull}(s_0) = 0$. All told, this exhibits the vectors
\[  (0, p\ell) \, , \, (e_1, p\ell)\, , \, \ldots\, , \, (e_d, p\ell)\, , \, (0, q)  \ \in \ \NN^{d+1} \]
as lying in $\Gamma_{Y\bull}(W\bull)$, which proves the lemma.
\end{proof}

\begin{remark} \label{Very.General.Flags.OK}
Observe for later reference that given countably many graded linear series $W_{\bullet, \alpha}$ each satisfying condition (B), there exists a flag $Y\bull$ for which the conclusion of the Lemma holds simultaneously for all of them.  In fact, since we are working over an uncountable base-field, one can fix a smooth point $y \in X$ at which the countably many birational morphisms defined by the relevant linear series $\linser{W_{k, \alpha}}$ are all defined and locally isomorphisms onto their images. Then as in the previous proof it suffices to  take any admissible flag $Y\bull$ with $Y_{d} = \{  y\}$.   \end{remark}

\begin{remark} \textbf{(Characteristic zero.)} If the ground field $\KK$ has characteristic zero, one can show with a little more effort that the  conclusion of the Lemma holds assumng only that   $\phi_m$ is generically finite over its image. \qed \end{remark}

We also give a criterion to guarantee that \eqref{full} holds with respect to any flag $Y\bull$.
\begin{definition} \textbf{(Condition (C)).} \label{Def.of.Cond.C}
Assume that $X$ is projective, and that $W\bull$ is a graded linear series associated to a big divisor $D$. 
We say that $W\bull$ satisfies condition (C) if:
\begin{enumerate}
\item [(i).] For every $m \gg 0$ there exists an effective divisor $F_m$ on $X$ such that the divisor\[ A_m \ =_{\text{def}} \ mD -F_m \] is ample; and  \sbl
\item[(ii).]
For all sufficiently large $p$,
\[
\HH{0}{X}{\OO_X(pA_m)} \ = \ \HH{0}{X}{\OO_X(pmD - pF_m)} \ \subseteq \ W_{pm} \ \subseteq \  \HH{0}{X}{\OO_X(pmD)}
\]
where the inclusion of the outer groups is the natural one determined by $pF_m$. 
\end{enumerate}
 \end{definition}

\begin{remark}  \textbf{(Alternate criterion for Condition (C).)}
\label{Alt.Cond.C}
As above, it is equivalent to ask that (i) and (ii) hold for one value $m_0$ of $m$, and that $W_k \ne 0$ for all $k \gg 0$. In fact, suppose that (i) and (ii) hold for $m = m_0$, and let $E_k$ be the divisor of a non-zero section $s_k\in W_k$: so $E_k \lin kD$, and multiplication by $s_k^{\otimes p}$ determines for any $\ell$ an embedding $W_{\ell} \subseteq W_{\ell + kp}$. Then (i) and (ii) hold for $m = m_0 + k$ by taking \[
F_{m_0+k} \ = \ F_{m_0} + E_k. \]
In fact, $A_m := (m_0+k)D - F_{m_0+k} \lin A_{m_0}$ is ample, and one has inclusions
\[
\HH{0}{X}{\OO_X(pA_{m_0}} \ \subseteq \ W_{pm_0} \ \subseteq \ W_{pm_0 + pk}. \qed
\]
\end{remark}


\begin{example} \textbf{(Restricted sections of a big divisor).} \label{Restricted.Volumes.Satisfy.Conditions}
An important situation where condition (C) holds involves restricted sections  of a big line divisor. This is discussed  in the next subsection. \qed
\end{example}

As suggested, condition (C) implies that \eqref{full} holds with respect to any flag. 
\begin{lemma} \label{Graded.Series.Any.Flag.C}
If $W\bull$ satisfies condition \tn{(C)}, then for any admissible flag $Y\bull$ on $X$,  the graded semigroup $\Gamma_{Y\bull}(W\bull)$ generates $\ZZ^{d+1}$ as a group.
\end{lemma}

\begin{proof}
Arguing as in the proof of Lemma \ref{Gamma.L.generates.Group}, it follows from the definition that for suitable $m$, and for any sufficiently large $ p \gg 0$, one can realize   in $\Gamma = \Gamma_{Y\bull}(W\bull)$ all the vectors
\[  (pf_m, pm) \ , \ (pf_m+ e_1, pm) \ , \ \ldots \ , \ (pf_m + e_d, pm) \ \in \ \NN^{d+1} ,\]
 where $f_m$ is the valuation vector of a section defining $F_m$ and $e_i \in \NN^d$ is the standard basis vector. Applying the definition a second time, one can find $q, \ell$ relatively prime to $m$ so that
$(qf_\ell, q\ell) \in \Gamma$ for some vector $f_\ell \in \NN^d$. The lemma follows. \end{proof}

Just as in the case of global linear series, one then arrives at:
\begin{theorem}  \label{Vol.OB.GLS}
Assume that $W\bull$ satisfies conditions \tn{(A)} and \tn{(B)}, or \tn{(C)}. Let $Y\bull$ an admissible flag as specified in \ref{Cond.F.generates.generic.flag} or \ref{Graded.Series.Any.Flag.C}. 
Then
\[   \vol_{\RR^d}\big (\Delta(W\bull)\big) \ = \ \frac{1}{d!}\cdot \vol(W\bull), \]
where \[\vol(W\bull) \ =_{\text{\tn{def}}} \ \lim_{m \to \infty} \frac{\dim W_m}{m^d/d!}.  \qed \]
\end{theorem}

\begin{remark} \textbf{(Volume of graded linear series as a limit).} The volume of a graded linear series  is usually defined to be the lim sup of the expression appearing in the Theorem. The fact that the limit exists assuming conditions (A) and (B) or (C) is new. \qed
\end{remark}

\begin{remark} {\textbf{(Complete linear series).}}
If $W\bull$ is the complete graded linear series associated to a big divisor $D$ on a   projective variety $X$ --- so that $W_m = \HH{0}{X}{\OO_X(mD)}$ --- then $W\bull$ satisfies condition (C) thanks to the basic properties of big divisors \cite[Chapter 2.2]{PAG}.  Hence the theory of Okounkov bodies $\Delta(D)$ for big $D$ is a special case of the more general picture for graded linear series. However in the interests of familiarity, we prefer to treat the classical case separately. \qed
\end{remark}

\subsection{Restricted Linear Series.}
\label{RLSS}


As they will come up on several occasions, we discuss briefly the graded linear series arising from restricted sections of a line bundle. 

We start by reviewing some definitions. Let $V$ be a projective variety, and let $D$ be a big divisor on $V$. Recall that the \textit{stable base locus} $\BBB(D) \subseteq V$  of $D$ is the intersection over all $m$ of the base loci $\Bs{mD}$. It is equivalent  to work with all sufficiently divisible $m$, so $\BBB(D)$ makes sense for any $\QQ$-divisor. (See \cite[Chapter 2.1]{PAG} for details and examples.) However this locus can behave somewhat unpredictably: for instance, it doesn't depend only on the numerical equivalence class of $D$. 
It turns out to be preferable to work instead with a variant obtained by perturbing $D$ slightly. Specifically, one defines the \textit{augmented base locus}  
$ \BBB_+(D) \ \subseteq  V$ to be:
\[  \BBB_+(D) \ = \ \BBB(D - A) \]
for any small ample $\QQ$-divisor $A$, this being independent of $A$ provided that it is sufficiently small. 
One can show that $\BBB_+(D)$ depends only on the numerical equivalence class of $D$, and so $\BBB_+(\xi)$ makes sense for any rational (or even real) numerical equivalence class $\xi$ on $V$. See \cite{ELMNP1} or \cite{ELMNP3} for details.

Now let $X \subseteq V$  be an irreducible subvariety of dimension $d$. Set
\[
W_m\ = \ \HH{0}{V|X}{\OO_V(mD)} \ =_{\text{def}}    \ \Image \Big (  \HH{0}{V}{\OO_V(mD)} \overset{\text{restr}} \lra \HH{0}{X}{\OO_X(mD)} \Big). 
\] 
These form a graded linear series on $X$ that we call the \textit{restricted complete linear series} of $D$ from $V$ to $X$. The \textit{restricted volume} of $D$ from $V$ to $X$ is by definition the volume of this graded series:
\[   \vol_{V|X}(D)   \ = \ \vol(W\bull). \] A detailed study of these restricted volumes appears   in the paper \cite{ELMNP3}.

\begin{lemma} \label{Restr.Lin.Series.Satisfies.Cond.C}
Assume that $X \not \subseteq \BBB_+(D)$. Then the restricted complete linear series $W\bull$ satisfies Condition \tn{(C)}. 
\end{lemma}

\begin{proof}
Let $A$ be a very ample divisor on $V$ which is sufficiently positive so that $A  + D$ is also very ample. By hypothesis $X \not \subseteq \BBB(D  - \eps A)$ for every sufficiently small rational $\eps >0$. This implies that there is some large integer $m_0 \in \NN$ such that $X \not \subseteq \Bs{m_0 D - A}$: so one can fix a divisor $E_{m_0} \in \linser{m_0 D   - A}$ that meets $X$ properly. Let $F_{m_0} = E_{m_0|X}$, and put $A_{m_0}  = (m_0D - F_{m_0})_{|X}$. Then $A_{m_0} \lin A_{|X}$ is an ample divisor on $X$. Moreover the natural map 
\[  \HH{0}{V}{\OO_V(pA)} \lra \HH{0}{X}{\OO_X(pA)} \]
is surjective when $p \gg 0$ thanks to Serre vanishing, which shows that \ref{Def.of.Cond.C}
(ii) holds for $m = m_0$. In view of Remark \ref{Alt.Cond.C}, it remains only to show that $\HH{0}{V|X}{\OO_V(mD)} \ne 0$ for $m \gg 0$. Clearly  $\HH{0}{V|X}{\OO_V(m_0D)} \ne 0$ since $A$ is very ample. On the other hand,
\[  (m_0 +1) D \ \lin \ (m_0 D - A) + (A+ D) , \]
and by construction the second term on the right is very ample. Therefore also \[ \HH{0}{V|X}{\OO_V((m_0 + 1)D)}  \ \ne \ 0 ,\] and since $\HH{0}{V|X}{\OO_V(mD)} \ne 0$ for two consecutive values of $m$, the group in question is non-zero for all $m \gg 0$.  
\end{proof}

We will denote by 
\[  \Delta_{V|X} (D) \ \subseteq \ \RR^d \]
the Okounkov body of $W\bull$ (with respect to a fixed admissible flag). Thus
\begin{equation} \label{OBVol.Compute.Restr.Vol}
\vol_{\RR^d}\big( \Delta_{V|X}(D)\big) \ = \ \frac{1}{d!} \cdot \vol_{V|X}(D). 
\end{equation}
The fact (coming from Theorem \ref{Vol.OB.GLS}) that the volume on the right is computed as a limit (rather than a limsup) was established in \cite[Cor. 2.15]{ELMNP3}.


\section{Fujita Approximations} \label{Fujita.Approximation.Section}

A very useful theorem of Fujita \cite{Fujita} (cf. \cite{DEL}, \cite[Chapter 11.4]{PAG}, \cite{Nakamaye}, \cite{Takagi}) asserts that the volume of any  big line bundle can be approximated arbitrarily closely by the self-intersection of an ample divisor on a modification. In this section we show how  the machinery developed so far can be used to give a new proof of this result, and extend it to the setting of graded linear series. As an application of this extension,  we also give a new proof of a result of the second author \cite{Mustata} concerning multiplicities of graded families of ideals, and establish the analogous statement on possibly singular varieties.

\subsection{Fujita's Approximation Theorem.}
We start with a variant of Proposition \ref{semigroup.proposition}. Specifically, consider again a sub-semigroup 
\[ \Gamma \ \subset \ \NN^{d+1}, \]
and define $\Delta = \Delta(\Gamma) \ \subseteq \RR^d$  and $\Gamma_m \subseteq \NN^d$ as in \eqref{K.Delta.Defn} and \eqref{Def.Gamma.m}.
\begin{proposition} \label{Semigroup.FAT}
Assume that $\Gamma$ satisfies conditions \eqref{zero} -- \eqref{full}, and fix $ \eps > 0$. There is an integer $p_0 = p_0(\eps)$ with the property that if $p \ge p_0$,  then 
\[
\lim_{k \to \infty} \ \frac{ \# \big (\, k *  \Gamma_p \,  \big)}{k^d p^d} \ \ge \ \vol_{\RR^d}(\Delta) - \eps,
\]
where the numerator on the left denotes the $k$-fold sum of points in $\Gamma_p$. 
\end{proposition}

\begin{lemma}
If $\Gamma \subseteq \NN^{d+1}$ is a semigroup that generates $\ZZ^{d+1}$ as a group, then $\Gamma_m \subseteq \NN^d$ generates $\ZZ^d$ as a group for all sufficiently large $m$.
\end{lemma}
\begin{proof} We may assume without loss of generality that $\Gamma$ is finitely generated. As in \eqref{K.Delta.Defn}, denote by $\Sigma = \Sigma(\Gamma) \subseteq \RR^{d+1}$  the closed convex cone generated by $\Gamma$.  The plan is to use again Khovanskii's result \cite[\S 3, Proposition 3]{Khovanskii} that 
\begin{equation} \label{Cone.In.Semigroup.Recall}
 \big( \Sigma \, + \, \gamma \big) \, \cap \, \NN^{d+1} \ \subseteq \ \Gamma \end{equation}
for a suitable $\gamma \in \Gamma$. 
Since $\Gamma$ generates $\ZZ^{d+1}$ the cone $\Sigma$ has non-empty interior, and hence so too does its unit slice $\Delta $. It follows that the set
\[
 \big( \Sigma \, + \, \gamma \big)_m \ =_{\text{def}} \big( \Sigma \, + \, \gamma \big) \, \cap \, \big ( \RR^d \times \{m \} \big)\ \subseteq \ \RR^d
\]
contains a ball of radius $> 2 \sqrt{d}$ provided that $m \gg 0$. But the integer points in any such ball span $\ZZ^d$. 
\end{proof}

\begin{proof}[Proof of Proposition \ref{Semigroup.FAT}] Assume to begin with that $\Gamma$ is finitely generated. Given $p$, let 
\[  \Theta_p \,  = \, \textnormal{convex hull} ( \, \Gamma_p \, )  \ \subseteq \ \RR^d. \]
It follows from the inclusion \eqref{Cone.In.Semigroup.Recall} that 
\[
\lim_{p \to \infty} \ \frac{\vol_{\RR^d}\big ( \Theta_p \big)}{p^d} \ = \ \vol_{\RR^d}\big( \Delta \big).
\]
On the other hand, since $\Gamma_p$ generates $\ZZ^d$ as a group for large $p$, we can apply \cite[\S 3, Corollary 1]{Khovanskii} , which states that
\[ 
\lim_{k \to \infty} \frac{ \# \big(  k * \Gamma_p \big)}{k^d} \ = \ \vol_{\RR^d}\big( \, \Theta_p \, \big).
\]
Putting these together, we find that given $\eps > 0$ there is an integer $p_0 = p_0(\eps)$ such that
\[
\lim_{k \to \infty} \frac{ \# \big(  k * \Gamma_p \big)}{p^dk^d} \ \ge \ \vol_{\RR^d}\big( \, \Delta \, \big) \, - \, \frac{\eps}{2} \]
when $p > p_0$. This gives what we want when $\Gamma$ is finitely generated. In the general  case, choose a finitely generated subsemigroup $\Gamma^\pr \subseteq \Gamma$ satisfying \eqref{zero} -- \eqref{full} such that $\vol(\Delta^\pr) \ge \vol(\Delta) - \eps/2$, and use the inequality just established for $\Gamma^\pr$. 
\end{proof}

Applying this in the global setting, we get a statement  essentially equivalent to the Fujita approximation theorem.

\begin{theorem} \label{FAT.Big.Divisors}
Let $D$ be a big divisor on an irreducible projective variety $X$ of dimension $d$, and for $p, k > 0$ write
\[
V_{k,p} \ = \ \Image \, \Big( S^k \HH{0}{X}{\OO_X(pD)} \lra \HH{0}{X}{\OO_X(pkD)} \Big).
\]  
Given $\eps > 0$, there exists an integer $p_0 = p_0(\eps)$ having the property that if $p \ge p_0$, then
\[
\lim_{k \to \infty}   \ \frac{ \dim V_{k,p} }{p^dk^d/d!} \ \ge \ \vol_X(D) - \eps. 
\]
\end{theorem}

\begin{remark}\textbf{(``Classical" statement of Fujita's theorem).}  \label{Explain.FAT.Traditional}
It may be worth explaining right away how \ref{FAT.Big.Divisors} implies more familiar formulations of Fujita's result, to the effect that one can approximate $\vol_X(D)$ by the volume of a big and nef (or even ample) divisor on a modification of $X$. Given $p$ such that $\linser{pD}$ is non-trivial, let \[
\mu : X^\pr =  X^\pr_p \lra X \]
be  the blowing-up of $X$ along the base-ideal $\frb\big( \linser{pD} \big)$, so that one can write
\[  
\mu^* \linser{pD} \ \subseteq \  \linser{M_p} + E_p,
\]
where $M_p$ is a basepoint-free divisor on $X^\pr$. Pullback of sections via $\mu$ determines a natural inclusion
\[  \Image \, \Big( S^k \HH{0}{X}{\OO_X(pD)} \lra \HH{0}{X}{\OO_X(pkD)} \Big) \ \subseteq \ \HH{0}{X^\pr}{\OO_{X^\pr}(kM_p)}. \]
Thus the theorem implies that
\[  \vol_{X^\pr}\big(\,  \tfrac{1}{p} M_p\, \big) \ \ge \ \vol_X(D) \, - \, \eps \]
when $p \ge p_0(\eps)$, 
which is one of the traditional statements of the result. \qed
\end{remark}

\begin{proof}[Proof of Theorem \ref{FAT.Big.Divisors}]
Fix any admissible flag $Y\bull$ on $X$, and consider the graded semigroup $\Gamma = \Gamma(D) $ of $D$ with respect to the corresponding valuation $\nu = \nu_{Y \bull}$. Thus $\Gamma_p$ consists precisely of the valuation vectors of non-zero sections of $\OO_X(pD)$:
\[  \Gamma_p \ = \ \Image \Big ( \big( \, \HH{0}{X}{\OO_X(pD)} -\{0 \} \, \big) \overset \nu  \lra \NN^d \, \Big).\]
Given  non-zero sections $s_1,\ldots , s_k \in \HH{0}{X}{\OO_X(pD)}$ one has
\[ \nu( s_1 \cdot \ldots \cdot s_k) \ = \  \nu(s_1) + \ldots + \nu(s_k), \] and it follows that
\[ k * \Gamma_p \ \subseteq \  \Image \Big ( \big( \, V_{k,p} -\{0 \} \, \big) \overset \nu  \lra \NN^d \, \Big).\]
But recall (Lemma \ref{No.of.Val.Vectors})  that the dimension of any space $W$ of sections counts the number of valuation vectors that $W$ determines, and that $\vol_{\RR^d}\big(\Delta\big) = \vol_X(D)/d!$. So the theorem is a consequence of \ref{Semigroup.FAT}.
\end{proof}

One of the advantages of the present approach is that the same argument immediately yields a version of Fujita approximations for suitable graded linear series. 
\begin{theorem} \label{FAT.GLS}
Let $X$ be an irreducible variety of dimension $d$,  let $W\bull$ be a graded linear series associated to a divisor $D$ on $X$,  and write \[
V_{k,p} \ = \ \Image \, \big( S^k W_p \lra W_{kp} \big).
\]  
Assume that $W\bull$ satisfies conditions \tn{(A)} and \tn{(B)}, or \tn{(C)}, and fix $\eps > 0$. There exists an integer $p_0 = p_0(\eps)$ having the property that if $p \ge p_0$, then
\[
\lim_{k \to \infty}   \ \frac{ \dim V_{k,p} }{p^dk^d/d!} \ \ge \ \vol_X(W\bull) - \eps.  \qed
\]
\end{theorem}

\begin{remark} \textbf{(Fujita approximation for restricted volumes).}
Thanks to Lemma \ref{Restr.Lin.Series.Satisfies.Cond.C}, this implies via the argument of 
\ref{Explain.FAT.Traditional}
 the main Fujita-type results for restricted volumes established in \cite[\S2]{ELMNP3}, notably the first equality of Theorem 2.13 of that paper. \qed \end{remark}

\begin{remark}
Note that we have not used here the hypothesis that our ground field $\KK$ is uncountable, so the results of this subsection (and the next) hold for varieties over an arbitrary algebraically closed field. \qed
\end{remark}

\subsection{Application to   Multiplicities of Graded Families of Ideals.} As an application of Theorem \ref{FAT.GLS}, we extend (in the geometric setting) the main result  of \cite{Mustata} to the case of possibly singular varieties. 

Let $X$ be an irreducible variety of dimension $d$. Recall that a \textit{graded familty of ideals} $\fra\bull = \{ \fra_k \}$ on $X$ is a family of ideal sheaves $\fra_k \subseteq \OO_X$, with $\fra_0 = \OO_X$, such that 
\[   \fra_k \cdot \fra_\ell \ \subseteq \ \fra_{k + \ell} \]
for every $k, \ell \ge 0$. A typical example occurs by taking a $\ZZ$-valued valuation $v$ centered on $X$, and setting 
\[  \fra_k \ = \ \big \{  f \in \OO_X \, \big | \,  v(f) \ge k \, \big \}. \]
We refer to \cite[Chapter 2.4]{PAG} for further discussion and illustrations. 

Now fix a point $x \in X$ with maximal ideal $\frakm$, and consider a graded family $\fra\bull$ with the property  that each $\fra_k$ is $\frakm$-primary.\footnote{Recall that this is equivalent to asking that each $\fra_k$ vanishes only at $x$.} Then $\fra_k \subseteq \OO_X$ is of finite codimension, and we define
\[  \mult\big (\fra\bull \big)  \ = \ \limsup_{m \to \infty}\  \frac{\dim_{\KK} \big( \OO_X / \fra_m \big)}{m^d/d!}. \] This is the analogue for $\fra\bull$ of the Samuel multiplicity of an ideal, and it is natural to ask how this invariant compares with the multiplicities $e(\fra_p)$ of the individual $\fra_p$. We prove
\begin{theorem} \label{Mult.Vol.Sing.Var}
One has 
\[   \mult \big (\fra\bull \big)  \ = \ \lim_{p \to \infty} \, \frac{e(\fra_p)}{p^d}. \]
\end{theorem}
\noi This was established in \cite{ELS} when $\fra\bull$ is the family of valuation ideals associated to an Abhyankar valuation centered at a smooth point of $X$. 
For an arbitrary $\frakm$-primary graded family in a regular local ring containing a field, the equality was proven by the second author in \cite{Mustata} via a degeneration to monomial ideals. It was suggested in \cite[p. 183]{PAG} that Theorem \ref{Mult.Vol.Sing.Var} should hold also at singular points.\footnote{The invariant $\mult(\fra\bull)$ was called the volume $\vol(\fra\bull)$ of $\fra\bull$ in \cite{ELS} and \cite{Mustata}, but we prefer to stick with the terminology used in  \cite{PAG}. }

The plan is to reduce to the case when $X$ is projective. The following lemma will then allow us to relate the local question at hand to global data.

\begin{lemma} \label{Proj.GSI.Lemma}
Let $X$ be a projective variety, and let $\fra\bull$ be a graded family of $\frakm$-primary ideals. Then there exists an ample divisor $D$ on $X$ with the property that for every $p, k > 0$, one has
\begin{equation} \label{GSI.Lemma.Van.Eqn}
 \HH{i}{X}{\OO_X(pkD) \otimes \fra_p^k} \, = \, 0 \ \text{for } i > 0.
 \end{equation}
Moreover we can arrange that the rational mapping \[ \phi_p \, : \, X \dra \PP = \PP\HH{0}{X}{\OO_X(pD)} \]
  defined by the subspace $\HH{0}{X}{\OO_X(pD)\otimes \fra_p} \subseteq \HH{0}{X}{\OO_X(pD)}$ is birational over its image.
\end{lemma}

\begin{proof}
By the definition of a graded family one has $\fra_1^{kp} \subseteq \fra_p^k$, and since $\fra_p^k / \fra_1^{kp}$ has zero-dimensional support,   the map
$\HH{i}{X}{\OO_X(kpD) \otimes \fra_1^{pk}} \lra\HH{i}{X}{\OO_X(kpD) \otimes \fra_p^{k}}  $ is surjective when $i > 0$. So it suffices to prove the vanishing \eqref{GSI.Lemma.Van.Eqn}
 in the case $p =  1$. For this, let 
\[  \mu  \, : \, X^\pr\, = \, \Bl_{\fra_1}(X) \lra X \]
be the blowing-up of $X$ along $\fra_1$, with exceptional divisor $E\subseteq X^\pr$. Let $D_0$ be an ample divisor on $X$. Since $-E$ is ample for $\mu$, we can suppose upon replacing $D_0$ by a large multiple that $\mu^*mD_0 - E$ is an ample divisor on $X^\pr$ for every $m \ge 1$. Recalling that 
\[  \mu_{*}\big(  \OO_{X^\pr}(- kE ) \big) \, = \, \fra_1^k \ \ , \ \  R^j\mu_{1} \big(\OO_{X^\pr}(- kE)\big) \, = \, 0 \ \ ( j > 0) \]
provided that $k \gg 0$, (cf. \cite[Lemma 5.4.24] {PAG}), it follows from  Fujita vanishing (\cite[]{PAG}) on $X^\pr$ and the Leray spectral sequence  that   \[
\HH{i}{X}{\OO_X(kmD_0) \otimes \fra_1^k} \ = \ 0  \ \ (i >0)  \]
 for every $m \ge 1$ and all sufficiently large $k$. By taking $m \ge m_1$ for suitable $m_1 > 0$ we can arrange that the vanishing in question holds for every $k$. So the first assertion of the Lemma will be satisfied with $D = mD_0$ for any choice of $m \ge m_1$.  Thanks to the inclusion
\[ 
\HH{0}{X}{\OO_X(pD)\otimes \fra_1^{p}} \ \subseteq \ \HH{0}{X}{\OO_X(pD)\otimes \fra_p},
\]
the  birationality of $\phi_p$ for arbitrary $p$ is implied by the case $p = 1$, and this can be achieved by increasing $m_1$.
 \end{proof}

\begin{proof}[Proof of Theorem \ref{Mult.Vol.Sing.Var}] By passing first to an affine neighborhood of $x$, and then taking a projective closure, we may assume without loss of generality that $X$ is projective. So we are in the setting of the previous lemma. Let $D$ to be the ample divisor constructed there, and set
\[  W_m \ = \ \HH{0}{X}{\OO_X(mD)\otimes \fra_m}. \]
These form a graded linear series associated to $D$, which satisfies condition (B) thanks to Lemma \ref{Proj.GSI.Lemma}.
Therefore Theorem \ref{FAT.GLS} applies.\footnote{Recall that condition (A) is automatic on a projective variety.}  Keeping the notation of that theorem, put
\[
V_{k,p} \  = \ \Image \Big(S^k \big ( \, \HH{0}{X}{\OO_X(pD) \otimes \fra_p}\, \big) \lra \HH{0}{X}{\OO_X(kpD) \otimes \fra_{pk}} \Big).
\]
The map on the right factors through $\HH{0}{X}{\OO_X(kpD) \otimes \fra_{p}^k} $, and hence
\[ 
V_{k,p} \ \subseteq \  \HH{0}{X}{\OO_X(kpD) \otimes \fra_{p}^k}.
\]
Therefore the vanishing $\HH{1}{X}{\OO_X(kpD) \otimes \fra_p^k} = 0$ from \ref{Proj.GSI.Lemma} gives
\begin{align*}
\dim \, V_{k,p} \ &\le \    \hh{0}{X}{\OO_X(kpD) \otimes \fra_{p}^k} \\
&= \ \hh{0}{X}{\OO_X(kpD) } \, - \, \dim \, \big( {\OO_X}/{\fra_p^k} \big).
\end{align*}
Thus 
\[\lim_{k \to \infty}   \ \frac{ \dim V_{k,p} }{p^dk^d/d!} \ \le \ \vol_X(D) - \frac{e(\fra_p)}{p^d}.
\]
On the other hand, Lemma \ref{Proj.GSI.Lemma}  similarly implies that
\[
\vol(W\bull) \ = \ \vol_X(D) \, - \, \mult(\fra\bull).
\] 
We deduce that given $\eps > 0$, there exists $p_0 = p_0(\eps)$ such that 
\[
\frac{e(\fra_p)}{p^d} \ \le \ \mult(\fra\bull) + \eps
\]  
for $p \ge p_0$. Since in any event $e(\fra_p)/p^d \ge \mult(\fra\bull)$, the Theorem follows. 
\end{proof}

\section{Variation of Okounkov Bodies} \label{Varation.of.OB.Section}

In this section we study the variation of $\Delta(D)$ as a function of the divisor $D$. We start by showing that $\Delta(D)$ depends only on the numerical equivalence class of $D$, and that it scales linearly with $D$. Therefore $\Delta(\xi)$ is naturally defined for any numerical equivalence class $\xi \in N^1(X)_{\QQ}$. The main result appears in the second subsection, where we prove Theorem \ref{Global.Body.Intro.Theorem} stated in the Introduction, showing that these occur as fibres of a closed convex cone \[
\Delta(X) \ \subseteq \ \RR^d \times N^1(X)_{\RR}. \]
 This is generalized to the setting of graded linear series in \S \ref{Multi-graded.Lin.Series.Subsection}:  the key here is to study linear series with an $\NN^r$-grading. Finally, we discuss  slices of Okounkov bodies in \S \ref{Slices.Subsection}, proving Corollary \ref{Integration.Intro}.

\subsection{Okounkov Body of a Rational Class.} 
Let $X$ be  an irreducible projective variety of dimesion $d$, and fix any admissible flag $Y\bull $ on $X$ with respect to which all the Okounkov bodies are constructed.
\begin{proposition} \label{Num.Nature.Homogenity} 
 Let $D$ be a big divisor on $X$.
\begin{enumerate}
\item[(i).] The Okounkov body $\Delta(D)$ depends only on the numerical equivalence class of $D$.
\item[(ii).]  For any integer $p >0$, one has
\[  \Delta(pD)  \ = \ p \cdot \Delta(D), \]
where the expression on the right denotes the homothetic image of $\Delta(D)$ under scaling by the factor $p$. 
\end{enumerate}
\end{proposition}

\begin{proof}
For (i), we need to show that $\Delta(D+P) = \Delta(D)$ for any numerically trivial divisor $P$. Arguing as in \cite[Lemma 2.2.42]{PAG}, there exists a fixed divisor $B$ such that $B + kP$ is very ample for every $k \in \ZZ$.\footnote{The arguments in \cite{PAG} rely on Fujita's vanishing theorem, which is valid in all characteristics.}  Choose a large integer $a$ such that $aD -B \lin F$ for some effective divisor $F$, and write
\[
(m+a)\big( D +  P \big)  \ \lin \ mD + (aD - B) + \big(B + (m+a)P \big).
\]
Upon representing $B + (m+a)P$ by a divisor not  passing through any of the subvarieties $Y_i $ in the flag $Y\bull$, one finds for all $m$ an inclusion
\[  \Gamma(D)_m  \, + \, f  \ \subseteq \  \Gamma(D+P)_{m+a},\]
where $f$ is the valuation vector of the section defining $F$. Letting $m \to \infty$ it follows that $\Delta(D) \subseteq \Delta(D+P)$. Replacing $D$ by $D+P$ and $P$ by $-P$ yields the reverse inclusion. For (ii), one argues as in the proof of \cite[Lemma 2.2.38]{PAG}. Specifically, choose an integer $r_0$ such that $\linser{rD} \ne\emptyset$ for $r > r_0$, and take $q_0$ with $q_0 p - (r_0 + p) >r_0$. Then for each $r \in [r_0 + 1, r_0 + p]$ we can fix effective divisors
\[  E_r \in \linser{rD} \ \ , \ \  F_r \in \linser{(q_0 p - r)D}. \]
This gives rise for every $r \in [r_0 + 1, r_0 +p]$ to inclusions
\[  \linser{mpD} + E_r + F_r \ \subseteq \ \linser{(mp +r)D}  + F_r\ \subseteq \ \linser{(m + q_0)pD}, \]
and hence also
\[ \Gamma(pD)_m + e_r + f_r \ \subseteq \ \Gamma(D)_{mp + r} + f_r \ \subseteq \ \Gamma(pD)_{m+q_0}\]
where $e_r$ and $f_r$ denote respectively the valuation vectors of $E_r$ and $F_r$. 
Letting $m \to \infty$ this gives \[
\Delta(pD) \ \subseteq \ p \cdot \Delta(D) \  \subseteq \ \Delta(pD), \]
as required.
\end{proof}

\begin{remark}
The homogenity $\Delta(pD) = p \cdot \Delta(D)$ is actually a consequence of Theorem \ref{Global.Okounkov.Body.Theorem} below, but it seems clearest for the development to establish it directly. \qed
\end{remark}

It follows from Lemma  \ref{Num.Nature.Homogenity} that the Okounkov body
\[  \Delta(\xi)  \ \subseteq \ \RR^d \]
is defined in a natural way for any big rational numerical equivalence class $\xi \in N^1(X)_\QQ$:
\begin{definition}\textbf{(Rational classes).} \label{Def.of.OB.for.Rat.Class}
Given a big class $\xi \in N^1(X)_\QQ$, choose any $\QQ$-divisor $D$ representing $\xi$, and  fix an integer $p \gg 0$ clearing the denominators of $D$. Then  set 
\[   \Delta(\xi) \ = \ \frac{1}{p} \cdot \Delta (pD) \  \subseteq \ \RR^d. \ \  \qed \] 
\end{definition}
 \noi The Lemma implies that this is independent of the choice of $D$ and $p$. Furthermore, the analogue of Theorem \ref{Vol.eq.Vol}
 remains valid:
 \begin{proposition} For any big class $\xi \in N^1(X)_{\QQ}$, one has 
  \[   \vol_{\RR^d}\big(\Delta(\xi) \big)\ = \ \frac{1}{d!} \, \cdot \, \vol_X(\xi). \]
 \end{proposition}
 
\begin{proof}
In fact, choose a $\QQ$-divisor $D$ representing $\xi$ and an integer $p \gg 0$ clearing the denominators of $D$.  Then
$ \vol_X(\xi) = \frac{1}{p^d} \cdot \vol_X(pD)$ by definition. Since likewise $\vol_{\RR^d}(\Delta(\xi)) = \frac{1}{p^d} \vol_{\RR^d}(\Delta(pD))$, the assertion follows from Theorem \ref{Vol.eq.Vol}
 \end{proof}

\subsection{Global Okounkov Body} \label{Glob.Ok.Body.Subsection}
We now show that the convex bodies $\Delta(\xi)$ fit together nicely. As above, $X$ is an irreducible projective variety of dimension $d$, and we fix an admissible flag $Y\bull$ on $X$ with respect to which all the constructions are made.
\begin{theorem} \label{Global.Okounkov.Body.Theorem}
There exists a closed convex cone 
\[ \Delta(X) \     \subseteq \ \RR^d \times N^1(X)_{\RR}\] characterized by the property that in the diagram 
 \[
\xymatrix{
\Delta(X)\ar[dr]      & \subseteq & \RR^d \times N^1(X)_{\RR} \ar[dl]^{\tn{pr}_2}  \\
&  N^1(X)_{\RR},
}
\]
 the fibre 
of $\Delta(X)$ over any big class $\xi \in N^1(X)_{\QQ}$ is $\Delta(\xi) $, i.e. 
 \[ \tn{pr}_2^{-1}(\xi) \cap \Delta(X)  \ = \ \Delta(\xi) \    \subseteq \ \RR^d \times \{\xi\} = \RR^d.\] 
\end{theorem}
\noi We emphasize that $\Delta(X)$ depends on the flag $Y\bull$, and we write $\Delta_{Y\bull}(X)$ when we wish to stress this dependence. Note also that $\Delta(X)$ is not a convex body but rather a closed convex cone in the vector space $\RR^d \times N^1(X)_{\RR}$. Nonetheless we will generally  refer to it as the global Okounkov body of $X$ (with respect to the given flag). We recall that the situation is illustrated schematically in Figure  
\ref{Global.Body.Picture} appearing in the Introduction.

To prove the theorem, the plan is to adapt the constructions of Section \ref{Okounkov.Construction.Section} to the multigraded setting. In order that we can limit ourselves to $\NN^r$-gradings, we start with a lemma about the pseudo-effective cone on a projective variety.

\begin{lemma}
Let $X$ be an irreducible projective variety of dimension $d$. Then the pseudo-effective cone $\overline{\tn{Eff}}(X)$ of $X$ is pointed, i.e. if $0 \ne \xi \in \overline{\tn{Eff}}(X)$ then $- \xi \not \in \overline{\tn{Eff}}(X)$.
\end{lemma}

\begin{proof}
We proceed by induction on $d$.
If $d=1$, then the assertion is trivial, while if $d=2$, it follows from the fact that the effective cone is the dual of the nef cone, which has full dimension. Suppose then that
$d\geq 3$;  we need to show that if $\xi$, $-\xi\in\Effbar(X)$, then $(\xi\cdot C)=0$ for every irreducible curve $C$ on $X$. Arguing by  induction, it is enough to show that there is an irreducible hypersurface $Y\subset X$ containing $C$ such that $\xi\vert_Y$, $-\xi\vert_Y\in\Effbar(Y)$. 
To this end,  write \[
\xi \ = \  \lim_{m\to\infty}d_m \ = \ -  \lim_{m\to\infty}e_m, \] where $d_m$ and $e_m$
are the classes of effective $\RR$-divisors $D_m$ and $E_m$ on $X$.  It is enough to find  a divisor $Y$  containing $C$ and not contained in
the support of any $D_m$ or $E_m$. But since $d \ge 3$ and we are working over an uncountable ground field, one can just take $Y$ to be a very general element of a linear series of  suitably ample divisors passing through $C$. 
\end{proof}

\begin{remark}
The Lemma remains valid for varieties over an arbitrary algebraically closed field. In fact, given such a variety one can always extend the ground field to an uncountable one without changing $\overline{\tn{Eff}}(X)$. \qed
\end{remark}

Returning to the construction of $\Delta(X)$, fix divisors $D_1, \ldots, D_r $ on $X$ whose classes form a $\ZZ$-basis of $N^1(X)$. Thanks to the previous lemma, we may -- and do -- choose the $D_i$ in such a way that every effective divisor  on $X$ is numerically equivalent to an $\NN$-linear combination of the $D_i$. 
 The choice of the $D_i$ determines  identifications 
\[  N^1(X)   \ = \ \ZZ^r \ \ \ , \ \ \ N^1(X)_{\RR} \ = \ \RR^r \]
 which we henceforth use without further comment. Observe  that under this isomorphism, the pseudo-effective cone $\overline{\tn{Eff}}(X)$ lies in the positive orthant of $\RR^r$. Given a vector $\Vec{m} = (m_1, \ldots , m_r) \in  \NN^r$, we write $\Vec{m} D = m_1 D_1 + \ldots + m_r D_r$.

We start by extending Definition \ref{Graded.Semigroup.of.Divisor}:
\begin{definition}
The \textit{multigraded semigroup} of $X$  (with respect to the fixed divisors $D_i$) is the additive sub-semigroup of $\NN^{d+r} = \NN^d \times \NN^r$ given by
\[
\Gamma(X) \ = \ \Gamma(X; D_1, \ldots , D_r) \ = \ \big \{ \, \big(\nu(s), \Vec{m} \big) \, \big | \, 0 \ne s \in \HH{0}{X}{\OO_X(\Vec m D)}\, \big \}. \ \qed
\]
\end{definition}
\noi Here of course the valuation $\nu$ is the one determined by the fixed admissible flag $Y\bull$.

Now denote by 
$ \Sigma(X)=  \Sigma(\Gamma  )  \subseteq  \RR^{d+r} $
the closed convex cone spanned by $\Gamma(X)$. Then we simply take
\begin{equation}
\Delta(X) \ = \ \Sigma(X) \ \subseteq \ \RR^{d} \times \RR^r.
\end{equation}
Note that while the construction of $\Delta(X)$ involves the choice of the divisors $D_1, \ldots, D_r$, it follows from Theorem \ref{Global.Okounkov.Body.Theorem} that after the identification $N^1(X)_{\RR} = \RR^d$ determined by the $D_i$, it is intrinsically defined. For the proof of the theorem, the essential point  will be to show that if  $\Vec{a}$ is an integer vector such that $\Vec{a}D$ is big, then the fibre of $\Sigma(X)$ over $\Vec{a} \in \RR^r$ coincides with $\Delta(\Vec{a}D)$. 
As in Section \ref{Volume.Ok.Body.Section}, it will be convenient to deduce this from some general statements about sub-semigroups of $\NN^d \times \NN^r$.

Consider then an additive semigroup 
\[ \Gamma\  \subseteq \ \NN^{d+r}   = \ \NN^d \times \NN^r, \]
and denote by 
\[  \Sigma \ = \ \Sigma(\Gamma) \ \subseteq \ \RR^{d+r} \]
the closed convex cone that it generates.  Define the \textit{support}
\[ \Supp(\Gamma) \ \subseteq \ \RR^{r} \]
of $\Gamma$ to be the image of $\Sigma$ under the projection to $\RR^r$: this is the same as the closed convex cone spanned by the image of $\Gamma$ under the projection $\NN^{d} \times \NN^r \lra \NN^r$.\footnote{Observe for this that the image of $\Sigma$ in $\RR^r$ is closed since $\Sigma $ -- being a pointed cone -- can be realized as the cone over a convex compact set.} Finally, given a vector $\Vec{a} \in \NN^r$, set
\begin{gather*}
\Gamma_{\NN \Vec{a}} \ = \ \Gamma \, \cap \, \big(\, \NN^d \times \NN \Vec{a}\, \big ) \\
\Sigma_{\RR \Vec{a}} \ = \ \Sigma(\Gamma)_{\RR \Vec{a}} \ = \ \Sigma \, \cap \, \big( \, \RR^d \times \RR \Vec{a} \, \big)
\end{gather*}
We view $\Gamma_{\NN\Vec{a}}$ as a sub-semigroup of $\NN^d \times \NN\Vec{a}= \NN^{d+1}$, and we denote by
\[ 
\Sigma(\Gamma_{\NN\Vec{a}} ) \ \subseteq \RR^d \times \RR\Vec{a} 
\]
the closed convex cone that it spans. 

\begin{proposition} \label{TTT.PPP}
Assume that $\Gamma$ generates a subgroup of finite index in $\ZZ^{d+r}$, and let $\Vec{a} \in \NN^r$ be a vector lying in the interior of $\Supp(\Gamma)$. Then
\[\Sigma(\Gamma_{\NN\Vec{a}})\ = \  \Sigma(\Gamma)_{\RR\Vec{a}}\]
\end{proposition}

\begin{remark}
The assumption on $\Gamma$ is equivalent to asking that $\Sigma$ have non-empty interior in $\RR^{d+r}$. Note that
the statement can fail if $\Vec{a} \not \in \tn{int}( \Supp (\Gamma))$: for instance, it could happen that $\Gamma_{\NN \Vec{a}} = \emptyset$, while $\dim \Sigma_{\RR\Vec{a}} > 0$.  \qed
\end{remark}

\begin{proof} [Proof of Proposition \ref{TTT.PPP}]
This is a special case of the results from Appendix \ref{Appendix.Section}. In fact, let $p : \RR^{d+r} \lra \RR^r$ denote the projection, and set $L = \RR\cdot \Vec{a} \subseteq \RR^r$. The assumption on $\Vec{a}$ implies that $L$ meets the interior of $p(\Sigma)$. Moreover,
\[
\Gamma\cap p^{-1}(L) \ = \Gamma_{\NN\Vec{a}} \ \ , \ \ \Sigma(\Gamma)_{\RR\Vec{a}} \ = \ \Sigma \, \cap \, p^{-1}(L).
\]
So the equality in the Proposition is exactly the assertion of Corollary \ref{Mircea.Corollary}.
\end{proof}

Returning to the setting of Theorem \ref{Global.Okounkov.Body.Theorem}, we start by showing that $ \Gamma(X; D_1, \ldots, D_r)$ verifies the hypothesis of the Proposition. 

\begin{lemma} \label{Global.SG.Gens.Grp}
The semigroup $\Gamma(X) \subseteq \NN^{d+r}$ generates $\ZZ^{d+r}$ as a group. 
\end{lemma}
\begin{proof} 
 Since the big cone $\tn{Big}(X)$ is an open subset of $N^1(X)_{\RR}$, there exist big divisor classes $e_1, \ldots, e_r \in N^1(X)$    spanning that free $\ZZ$-module. The conditions on   $D_1, \dots, D_r$ imply that each $e_j$ is an $\NN$-linear combination of (the classes of) the $D_i$, say $e_j \num \Vec{a}_j D$ for some $\Vec{a}_j \in \ZZ^r$.   Set \[
E_j \ = \  \Vec{a}_j D   \]
(as divisors). Then the graded semigroups $\Gamma(E_j)$ sit in a natural way as sub-semigroups of $\Gamma(X)$, and  Lemma \ref{Gamma.L.generates.Group} shows that 
 that $\Gamma(E_j)$ generates $\ZZ^d \times \ZZ \cdot \Vec{a_j}$ as a group. The Lemma then follows from the fact that  $\Vec{a}_1, \ldots, \Vec{a}_r$ span $\ZZ^r$.  \end{proof}

 \begin{proof}[Proof of Theorem \ref{Global.Okounkov.Body.Theorem}] 
Set $\Gamma = \Gamma(X; D_1, \ldots, D_r)$. Then the support of $\Gamma$ consists of the closed cone spanned by all vectors $\Vec{a} \in \ZZ^r = N^1(X)$ such that $\HH{0}{X}{\OO_X(\Vec{a}D)} \ne 0$: this is the pseudo-effective cone $\overline{\tn{Eff}}(X)$ of $X$,  whose interior is the big cone $\tn{Big}(X)$ (cf. \cite[Chapter.2.2.B]{PAG}).  So $\Vec{a} \in \tn{interior}(\Supp(\Gamma))$ if and only if $\OO_X(\Vec{a}D)$ is big. Given such a vector $\Vec{a}$, it follows from the definitions that
\[  \Gamma(X)_{\NN\Vec{a}} \ = \ \Gamma(\Vec{a}D)\ \subseteq \ \NN^d \times \NN\Vec{a}, \]
and hence the Okounkov body $\Delta(\Vec{a}D)$ is the base of the cone
$\Sigma(\Gamma_{\NN\Vec{a}})$, i.e.
\begin{equation}
\Delta(\Vec{a}D) \ = \ \Sigma\big(\Gamma_{\NN\Vec{a}}\big) \cap \big( \RR^d \times \{\Vec{a} \}\big).
\end{equation}
But the Proposition implies that this coincides with the fibre $\Delta(X)_{\Vec{a}}$ of $\Delta(X)$ over $\Vec{a} \in \RR^d$, which verifies the Theorem for integral vectors $\Vec{a}$. The case of rational classes follows since both sides of the desired equality $\Delta(\xi) = \Delta(X)_{\xi}$ scale linearly with $\xi$. 
\end{proof}

As noted in the Introduction, the Theorem implies some basic properties of the volume function:

\begin{corollary}
There is a uniquely defined continuous function
\[
\vol_X : \tn{Big}(X) \lra \RR 
\]
that computes the volume of any big rational class. This function  is homogeneous of degree $d$, and log-conave, i.e.
\begin{equation} \vol_X(\xi + \xi^\pr)^{1/d} \ \ge \ \vol_X(\xi)^{1/d} + \vol_X(\xi^\pr)^{1/d} \tag{*} \end{equation}
for any $\xi , \xi^\pr \in \tn{Big}(X)$.
\end{corollary}
\begin{proof} One takes of course
\[
\vol_X(\xi) \ = \ d! \cdot \vol_{\RR^d}\big( \Delta(\xi) \big), 
\]
where $\Delta(\xi) = \Delta(X)_\xi$   is the fibre of the projection  $\Delta(X) \lra \RR^r$. Then the assertions are standard results from convex geometry. In fact, as explained in \cite[\S 5]{Ball} the convexity of $\Delta(X)$ implies that 
\[ \Delta(\xi) + \Delta(\xi^\pr) \ \subseteq \ \Delta(\xi + \xi^\pr),
\]
and so (*) follows from the Brunn-Minkowski theorem. In view of  the homogenity of $\vol_X$, (*) means that the function $\xi \to \vol_X(\xi)^{1/d}$ is concave. But any concave function is continuous on the interior of its domain (c.f. \cite[Theorem 2.2]{Gruber}), which gives the first statement of the Corollary.
\end{proof}

\begin{remark} It  was established in \cite[Corollary 2.2.45]{PAG}  that  $\vol_X$ actually extends to a continuous function on all of $N^1(X)_{\RR}$ that is zero outside $\tn{Big}(X)$. Besides the continuity appearing in the Corollary, this includes the assertion  that $\vol_X(\xi) \to 0$ as $\xi$ approaches a point $\xi_0 \in {\overline{\tn{Eff}}(X)}$ on the boundary of the pseudo-effective cone.  \qed
\end{remark}

\subsection{ Multi-Graded Linear Series.} 
\label{Multi-graded.Lin.Series.Subsection}
We now wish to extend the previous discussion to the setting of graded linear series. To this end, it is natural to work with $\NN^r$-graded linear series. 

We start with some definitions. Let $X$ be an irreducible variety of dimension $d$, and fix divisors $D_1, \ldots, D_r$ on $X$. For $\Vec{m} = (m_1,\ldots, m_r) \in \NN^r$ we write as above $\Vec{m}D = \sum m_i D_i$, and we put $|\Vec{m}| = \sum |m_i|$. 
\begin{definition}
A \textit{multigraded linear series} $W\vecbull$ on $X$ associated to the $D_i$ consists of finite-dimensional subspaces
\[  W_{\Vec{k}} \ \subseteq \ \HH{0}{X}{\OO_X(\Vec{k}D)}\]
for each $\Vec{k} \in \NN^r$, with $W_{\Vec{0}} = \KK$, such that
\[  W_{\Vec{k}} \cdot W_{\Vec{m}} \ \subseteq \ W_{\Vec{k} + \Vec{m}}. \]
As in the singly graded case, the multiplication on the left denotes the image of $W_{\Vec{k}} \otimes W_{\Vec{m}} $ under the natural map $\HH{0}{X}{\OO_X(\Vec{k}D) } \otimes \HH{0}{X}{\OO_X(\Vec{m}D) } \lra \HH{0}{X}{\OO_X((\Vec{k}+\Vec{m})D) } $.
\end{definition}
Given $\Vec{a} \in \NN^r$, denote by $W_{\Vec{a},\bullet}$ the singly graded linear series  associated to the divisor $\Vec{a}D$ given by the subspaces $W_{k \Vec{a}} \subseteq \HH{0}{X}{\OO_X(k\Vec{a}D)}$. Then put
\[ \vol_{W\vecbull}(\Vec{a}) \ = \ \vol(W_{\Vec{a},\bullet}) \]
(assuming that this quantity is finite). In this way a multi-graded linear series defines a volume function on $\NN^r$ (and later on $\RR^r$). Similarly, having fixed an admissible flag $Y\bull$ on $X$, write
$\Delta(\Vec{a}) = \Delta(W_{\Vec{a}, \bullet})$. Finally,   the \textit{support} \[ \Supp(W\vecbull) \ \subseteq  \ \RR^r \] of $W\vecbull$ is the closed convex cone spanned by all indices $\Vec{m} \in \NN^r$ such that $W_{\Vec{m}} \ne 0$.

We seek conditions on $W\vecbull$, extending those introduced in Section \ref{Volume.Ok.Body.Section}, to guarantee that these constructions work well. This is most easily achieved by reducing to the singly graded case. We start with:
\begin{definition} $W\vecbull$ satisfies Condition (B$^\pr$) (or Condition (C$^\pr$)) if the following hold:
\begin{enumerate}
\item[(i).]  $\Supp(W\vecbull) \subseteq \RR^r$ has non-empty interior;
\item[(ii).]  For any integer vector $\Vec{a} \in \interior (\Supp(W\vecbull))$, \[  W_{k\Vec{a}} \ \ne \ 0 \ \ \text{for} \ k \gg 0; \]
\item[(iii).] There exists an integer vector $\Vec{a}_0 \in  \interior\big(\Supp(W\vecbull)\big)$ such that the $\NN$-graded linear series $W_{\Vec{a}_0 ,\bull}$ satisfies Condition (B) (or Condition (C)). 
\end{enumerate}
\end{definition}
\noi Recall that Condition (C) (and hence also (C$^\pr$)) includes the requirement that $X$ be projective.

This definition implies that   the singly-graded linear series determined  by $W\vecbull$ have the corresponding property:
\begin{lemma}
Assume that $W\vecbull$ satisfies Condition \tn{(B$^\pr$)} or \tn{(C$^\pr$)}. If 
\[ \Vec{a}  \ \in \ \interior  \big(\Supp(W\vecbull)\big)
\] is any integer vector, then  $W_{\Vec{a},\bull}$ satisfies the corresponding condition $\tn{(B)}$ or $\tn{(C)}$.
\end{lemma}

\begin{proof}
We will write the proof for Condition (C$^\pr$), the case of (B$^\pr$) being similar but simpler. By definition, for any sufficiently large integer $m \gg 0$, there is an effective divisor $F_{m \Vec{a}_0 }$ such that
\begin{equation}m\Vec{a}_0 D - F_{m\Vec{a}_0} \ \lin \ A_{m\Vec{a}_0}  \notag
\end{equation}
is ample, and
\begin{equation} \HH{0}{X}{\OO_X(pA_{m\Vec{a}_0})} \ \subseteq \ W_{pm\Vec{a}_0} \ \subseteq \ \HH{0}{X}{\OO_X(pm\Vec{a}_0D)}. 
\tag{*} \end{equation}
Now let  $\Vec{a} \in \interior (\Supp(W\vecbull))$ be any integer vector. Then for some large $k \in \NN$,
\[ k\Vec{a} \ = \ \Vec{a}_0 + \Vec{b} \]
where $\Vec{b} $ also lies in the interior of $\Supp(W\vecbull)$. Therefore $W_{m\Vec{b}} \ne 0$ for $m \gg 0$: let $E_{m \Vec{b}} $ be the divisor of a non-zero section $s_{m \Vec{b}} \in W_{m \Vec{b}}$, so that $E_{m \Vec{b}} \lin m\Vec{b}D$. Then $mk\Vec{a}D = m\Vec{a}_0D + m\Vec{b}D$, and consequently
\[
mk\Vec{a}D - F_{m\Vec{a}_0} - E_{m\Vec{b}} \ \lin \ A_{m \Vec{a}_0} 
\]
is ample. Moreover, for all $p \gg 0$
\[  \HH{0}{X}{\OO_X(pA_{m \Vec{a}_0} )} \ \subseteq \ W_{pm\Vec{a}_0} \ \subseteq \ W_{pmk\Vec{a}}, \] the first inclusion coming from (*), and the second arising from  multiplication by $s_{m \Vec{b}}^{\otimes p}$.
This shows that $W_{\Vec{a}, \bullet}$ satisfies the two properties (i) and (ii) in Definition \ref{Def.of.Cond.C}
 for one value of the parameter appearing there, and then it follows from Remark \ref{Alt.Cond.C} that Condition (C) itself holds. 
\end{proof}

Now fix an admissible flag $Y\bull$ on $X$. For the boundedness questions, we propose:
\begin{definition}
$W\vecbull$ satisfies Condition (A$^\pr$)  with respect to $Y\bull$ if there is an integer $b \gg 0$ such that for every $\Vec{m} \in \NN^r$ and every $0 \ne s \in W_{\Vec{m}}$, 
\begin{equation} \nu_i(s) \ \le \ b \cdot |\Vec{m}| \end{equation}
for all $1 \le i \le d$. \qed
\end{definition}
\noi This evidently implies that any of the simply graded linear series $W_{\Vec{a}, \bullet}$ (for $\Vec{a} \in \NN^r$) satisfy Condition (A). 
Remark \ref{Extn.of.Bddness.Lemma.Remark}
 shows that it holds automatically when $X$ is projective. 

It follows from the Lemma and the results of Section \ref{Conds.GLS.subsection}  that if Conditions (A$^\pr$) and (B$^\pr$) or (C$^\pr$)  hold for $W\vecbull$, then with respect to a suitable flag $Y\bull$, the Okounkov bodies $\Delta(\Vec{a})$ are defined and compute $\vol_{W\vecbull}(\Vec{a})$ for every integer vector  $\Vec{a}$ lying in the interior of $\Supp(W\vecbull)$. Our next task is to realize these as the fibres of a global cone $\Delta(W\vecbull) \subseteq \RR^d \times \RR^r$. 

 Fix an admissible flag $Y\bull$ on $X$. The \textit{multi-graded semigroup} of $W\vecbull$  with respect to $Y\bull$ is defined to be 
 \[  \Gamma(W\vecbull) \ = \ \Gamma_{Y\bull}(W_{\vecbull})  \ = \ \big \{ \, (\nu(s), \Vec{m}) \, \big  | \, 0 \ne s \in W_{\Vec{m}} \, \big \}  \ \subseteq \ \NN^{d+r}. \]
\begin{lemma} \label{Good.Flags.Multigraded.Systems}
If $W\vecbull$ satisfies Condition \tn{(B$^\pr$)}, then there exists a flag $Y\bull$ for which $\Gamma_{Y\bull}(W_{\vecbull}) $ generates $\ZZ^{d + r}$ as a group. If $W\vecbull$ satisfies Condition \tn{(C$^\pr$)}, then the same statement holds for any admissible flag $Y\bull$. 
\end{lemma}

\begin{proof} Given an integer vector $\Vec{a} \in \NN^r$ lying in the interior of $\Supp(W\vecbull)$, denote by
\[  \Gamma_{\Vec{a}} \ = \ \Gamma_{Y\bull} (W_{\Vec{a}, \bullet}) \ \subseteq \ \NN^{d} \times \NN\Vec{a} \ \subseteq \ \NN^d \times \NN^r\]
the graded semi-group of $W_{\Vec{a}, \bullet}$ with respect to $Y\bull$, which is naturally a sub-semigroup of $\Gamma(W\vecbull)$.
Bearing in mind Remark \ref{Very.General.Flags.OK}, we can suppose  that  each $\Gamma_{\Vec{a}}$ generates $\ZZ^d \times \ZZ\Vec{a}$ as a group, and then the argument proceeds as in the Proof of Lemma \ref{Global.SG.Gens.Grp}. In fact, if we choose $\Vec{a_1}, \ldots, \Vec{a_r}$ spanning $\ZZ^r$, then the corresponding $\Gamma_{\Vec{a}_i}$ together generate $\ZZ^{d+r}$.\end{proof}

Now let
\[  \Sigma(W\vecbull) \ \subseteq \ \RR^d \times \RR^r \]
be the closed convex cone spanned by $\Gamma(W\vecbull)$, set \[  \Delta(W\vecbull) \ = \ \Sigma(W\vecbull), \]
and consider the diagram:
\[
\xymatrix{
\Delta(W\vecbull)\ar[dr]      & \subseteq & \RR^d \times \RR^r \ar[dl]^{\tn{pr}_2}  \\
&    \RR^r.
}
\]
Then just as in the global case, one has
\begin{theorem}
Assume that $W\vecbull$ satisfies Conditions \tn{(A$^\pr$)} and \tn{(B$^\pr$)}, or \tn{(C$^\pr$)}, and let $Y\bull$ be an admissible flag  as specified in Lemma \ref{Good.Flags.Multigraded.Systems}. Then for any integer vector 
\[ \Vec{a} \in \interior\big( \Supp (W\vecbull) \big), \] the fibre of $\Delta(W\vecbull)$ over $\Vec{a}$ is the corresponding Okounkov body of $W_{\Vec{a}, \bullet}$\tn{:}
\[  \Delta(W\vecbull)_{\Vec{a}} \ = \ \Delta(\Vec{a}).  \qed \]
\end{theorem}

\noi Note that it follows from the Theorem that 
\begin{equation} \Delta(p \Vec{a}) \ = \ p \cdot \Delta(\Vec{a}) \ \ \ 
\text{and}  \ \ \ \vol_{W\vecbull}(p\Vec{a}) \ =\  p^d \cdot \vol_{W\vecbull}(\Vec{a}).
\end{equation}
 (This can also be shown directly.) Therefore  $\Delta(\Vec{\alpha})$ and $\vol_{W\vecbull}(\Vec{\alpha})$ are naturally defined by homogeneity (as in \ref{Def.of.OB.for.Rat.Class}) for any rational vector $\Vec{\alpha} \in \QQ^r$ lying in the interior of $\Supp(W\vecbull)$, and one has
\begin{equation}
  \Delta(W\vecbull)_{\Vec{\alpha}} \ = \ \Delta(\Vec{\alpha}). \notag
\end{equation}

\begin{corollary} \label{Continuity.Abstract.Volume.Function}
Under the hypotheses of the Theorem,
the function $\Vec{a} \mapsto \vol_{W\vecbull}(\Vec{a})$ extends uniquely to a continuous function
\[
\vol_{W\vecbull} : \interior \big(\Supp(W\vecbull)\big) \lra \RR 
\]
which is homogeneous of degree $d$, and the resulting function is  log-concave. \qed
\end{corollary}

\begin{remark}
It is not hard to construct an example of a multigraded linear series $W\vecbull$, together with an integer vector $\Vec{a}$ lying on the boundary of $\Supp(W\vecbull)$, such that $W_{\Vec{a}, \bullet}$ is perfectly well-behaved -- e.g. satisfies Condition (C) -- but where nonetheless  $\vol_{W\vecbull}(\Vec{\alpha})$ does not converge to $\vol_{W\vecbull}(\Vec{a})$ as $\Vec{\alpha} \to \Vec{a}$. \qed
\end{remark}

\begin{example} \textbf{(Restricted volume function.)} \label{Restricted.Vol.Fn.Example}
Let $V$ be an irreducible projective variety, and fix as in \S \ref{Glob.Ok.Body.Subsection} divisors $D_1, \ldots, D_r$ on $V$ whose classes span $N^1(V)_{\RR}$. Given an irreducible subvariety $X \subseteq V$ of dimension $d$, consider the $\NN^r$-graded linear series $W\vecbull$ given by
\[
W_{\Vec{m}} \ = \ \HH{0}{V|X}{\OO_V(\Vec{m}{D})}.
\]
It follows from Lemma \ref{Restr.Lin.Series.Satisfies.Cond.C} that this satisfies condition (C$^\pr$), and
the interior of the support of $W\vecbull$ is the set
\[ 
\tn{Big}^+\big( V|X \big) \ = \ \big\{ \Vec{\alpha} \in \RR^r \, \big | \, X \not \subseteq \BBB_+(\Vec{\alpha}D) \, \big \}.
\]
It follows first of all that  we get a global Okounkov body, which one might denote by
\[  \Delta(V|X) \ \subseteq \ \RR^d \times \RR^r, \]
with fibre $\Delta_{V|X}(\Vec{\alpha})$ over $\Vec{\alpha} \in \tn{Big}^+(V|X)$. 
So Corollary \ref{Continuity.Abstract.Volume.Function} also
 yields the continuity and log-concavity of the restricted volume function
\[
\vol_{V|X} : \tn{Big}^+\big( V|X \big)  \lra \RR
\]
established in \cite[Theorem A]{ELMNP3}. Note however that one does not recover  the most substantial  result of that paper, namely  that $\vol_{V|X}(\xi) \to 0$ as $\xi \to \xi_0$ when $\xi_0$ is a class on the boundary of $\tn{Big}^+\big( V|X \big)$ such that $X$ is an irreducible component of $\BBB_+(\xi_0)$. \qed
\end{example}

\begin{remark} \label{Numerical.Nature.Restricted.OB}
In the situation of the previous Example, one can show as in the proof  of  Proposition \ref{Num.Nature.Homogenity}   that if $D$ is a big divisor on $V$ such that $X \not \subseteq \BBB_+(D)$, then the Okounkov body $\Delta_{V|X}(D)$ only depends on the numerical equivalence class of $D$. Therefore, as in the global setting it is meaningful to speak of $\Delta_{V|X}(\xi)$ for any numerical equivalence class $\xi \in \tn{Big}^+(V|X)$.   \qed  \end{remark}

\subsection{Slices of Okounkov Bodies.} \label{Slices.Subsection}

Let $X$ be an irreducible projective variety of dimension $d$, and let $E \subseteq X$ be an irreducible (and reduced) Cartier divisor on X.\footnote{Observe that the hypothesis that $E$ be Cartier implies in particular that $E$ is not contained in the singular locus of $X$.} In this subsection we study Okounkov bodies computed with respect to an admissible flag $Y\bull$ 
\begin{equation} \label{Flag.With.E}
 X   \ \supseteq \   E\  \supseteq \ Y_2 \ \supseteq \ \ldots \ \supseteq \ Y_{d-1} \  \supseteq \  Y_{d} \  = \  \{ \text{pt} \} \end{equation}
with divisorial component $Y_1 = E$. In particular, we prove Corollary \ref{Integration.Intro} from the Introduction.

Let $\xi \in N^1(X)_{\RR}$ be a big class, and consider the Okounkov body 
\[  \Delta(\xi)  \ = \ \Delta(X)_{\xi} \ \subseteq \ \RR^d \]
computed with respect to the flag $Y\bull$. 
Write $\pro_1 : \Delta(\xi) \lra \RR$ for projection onto the first coordinate, and set
\begin{align*} \Delta(\xi)_{\nu_1 = t}\ &= \ \pro_1^{-1}(t) \ \subseteq \ \{t \} \times \RR^{d-1}  \ = \ \RR^{d-1}
\\ \Delta(\xi)_{\nu_1\ge t} \ &=  \ \pro_1^{-1}\big( [t, \infty) \big) \ \subseteq \
 \RR^d. \end{align*}
Our purpose is to interpret these sets in terms of Okounkov bodies associated to the divisor class $\xi - te$, where $e \in N^1(X)$ is the class of $E$. 

We assume that $E \not \subseteq \BBB_+(\xi)$ (see \S \ref{RLSS}), which guarantees that  $\Delta(\xi)_{\nu_1 = 0} \ne \emptyset$. Put 
\begin{equation}  \mu(\xi;e) \ = \ \sup \, \{ \, s > 0 \  | \  \xi - s \cdot e \in \tn{Big}(X)\, \}.\tag{*}
\end{equation}
This invariant computes  the right-hand endpoint of the image of $\Delta(\xi)$ under the projection $\pro_1 : \RR^d \lra \RR$: one checks that $E\not\subseteq \BBB_+(\xi - se)$ when $0 \le s \le \mu(\xi,e)$.
\begin{theorem} \label{Slicing.Theorem}
Continue to assume that  $E \not \subseteq \BBB_+(\xi)$,  and fix any real number $t$ with $0 \le t < \mu(\xi; e)$. 
Then 
\begin{equation} \label{First.Slicing.Equation}
 \Delta(\xi)_{\nu_1 \ge t} \ = \ \Delta(\xi - te) + t \cdot \Vec{e}_1, \end{equation}
where $\Vec{e}_1 = (1, 0, \ldots, 0) \in \NN^d$ is the first standard unit vector. Furthermore,  
\begin{equation} \label{Second.Slicing.Equation}
 \Delta(\xi)_{\nu_1 = t} \ = \ \Delta_{X|E}(\xi - te\big).
 \end{equation}
\end{theorem}
\noi   Naturally enough, the Okounkov body $\Delta_{X|E}$   appearing in \eqref{Second.Slicing.Equation}
 is computed with respect to the flag
\[ Y\bull\, | \, E \ \  : \ \ \ E\  \supseteq \ Y_2\  \supseteq\  \ldots\  \supseteq \ Y_{d}\]
 on $E$. Note that \eqref{Second.Slicing.Equation} implicitly assumes the fact stated without proof in Remark \ref{Numerical.Nature.Restricted.OB}, that $\Delta_{X|E} $ is well-defined on numerical equivalence classes. However this is purely for typographical convenience: nothing would change by working with specific divisors in the theorem and the next corollary. 
We refer again to  Figure \ref{Slices.Picture}
 in the Introduction for a schematic illustration of the result.

 Before proving the theorem, we note some consequences:
\begin{corollary} \label{Cor.of.Slicing.Theorem}
Keep the assumptions of the theorem. 
\begin{itemize}
\item[(i).]
One has
\[
\vol_{\RR^{d-1}} \big( \Delta(\xi)_{\nu_1 = t}\big) \ = \ \frac{1}{(d-1)!} \cdot \vol_{X|E}(\xi - te).
\]
\item[(ii).]
For any  $0 < a < \mu(\xi;e)$,
 \[
\vol_X(\xi) - \vol_X(\xi - ae) \ = \ d \cdot \int_{-a}^{0} \vol_{X|E}(\xi + te) \,  dt.
\]
\item[(iii).] The function $t \mapsto \vol_X(\xi + te)$ is differentiable at $t = 0$, and
\[
\frac{d}{dt} \,\big( \vol_X(\xi  + te) \big) |_{t = 0} \ = \ d \cdot \vol_{X|E}(\xi).
\]
\end{itemize}
\end{corollary}

\begin{proof}[Proof of Corollary \ref{Cor.of.Slicing.Theorem}]  The first assertion follows from \eqref{Second.Slicing.Equation} and \eqref{OBVol.Compute.Restr.Vol}, and (ii) is the statement that one can compute the volume of $d$-dimensional convex body by integrating the $(d-1)$-dimensional volumes of the fibres of an orthogonal projection to $\RR$. For (iii), the one additional point to observe is that the hypothesis on $E$ implies that $E \not \subset \BBB_+(\xi + \eps e)$ for $0<\eps \ll 1$. Therefore we can apply the Theorem with $\xi$ replaced by $\xi + \eps e$, and then (ii) yields the two-sided differentiability  of $t  \mapsto \vol_X(\xi + te)$ at $t = 0$. 
\end{proof}

\begin{remark} \textbf{(Corollary C.)}
If $E$ is a very ample divisor that is general in its linear series, then 
  the condition $E \not \subseteq \BBB_+(\eta)$ holds whenever $\eta$ is big. Thus the theorem and the corollary reduce in this case to the statements appearing in the Introduction. \qed
\end{remark}

\begin{remark} \textbf{(Differentiability of volume.)} Continuing the train of thought of the previous remark, consider  a basis of $N^1(X)_{\RR}$ consisting of the classes of very ample divisors. Then statement (iii) of the corollary (together with Example \ref{Restricted.Vol.Fn.Example})  implies that the volume function $\vol_X$ has continuous partials in all directions at any point $\xi \in \tn{Big}(X)$, i.e. the function
\[
\vol_X : \tn{Big}(X) \lra \RR 
\]
is $\mathcal{C}^1$. Boucksom--Favre--Jonsson give a different proof of \ref{Cor.of.Slicing.Theorem} (ii) and (iii) in \cite{BFJ}, where they study in detail the differentiability properties of $\vol_X$ and its consequences. \qed
\end{remark}

Turning to the proof of the Theorem, one piece of notation will be helpful. Namely, given a graded semigroup $\Gamma \subseteq \NN^d \times \NN$, and an integer $a  > 0$, denote by $\Gamma_{\nu_1 \ge a} \subseteq \Gamma$ and $\Gamma_{\nu_1 = a} \subseteq \Gamma$ the sub-semigroups
\begin{align*}
\Gamma_{\nu_1 \ge a}  \ &= \ \big\{ (\nu_1, \ldots, \nu_d, m) \in \Gamma \, \big | \, \nu_1 \ge am \, \big \} \\
\Gamma_{\nu_1 = a}  \ &= \ \big\{ (\nu_1, \ldots, \nu_d, m) \in \Gamma \, \big | \, \nu_1 = am \, \big \}.
\end{align*}
\begin{proof}[Proof of Theorem \ref{Slicing.Theorem}] As in the proof of Corollary \ref{Cor.of.Slicing.Theorem}, it is enough to prove the Theorem when $t >0$ since we can replace $\xi$ by $\xi + \eps e$ for $0 < \eps \ll 1$ to get the original statement with $t = 0$.  As always, write $\nu = \nu_{Y\bull}$ for the valuation determined by $Y\bull$. 
For the first statement, consider to begin with an integral divisor $D$ and an integer $a > 0$ such that $D - aE$ is big. 
Then for any $m \ge 0$, $\HH{0}{X}{\OO_X(mD - maE)}$ sits naturally as a subgroup of $\HH{0}{X}{\OO_X(mD)}$, and in fact
\begin{align*}
 \HH{0}{X}{\OO_X(mD - maE)} \ &= \ \big\{ s \in \HH{0}{X}{\OO_X(mD)} \ \big | \ \ord_E(s) \ge ma \ \big\}. \\
 &= \ \big\{ s \in \HH{0}{X}{\OO_X(mD)} \ \big | \ \nu_1(s) \ge ma \ \big\}.
 \end{align*}
In view of the definition of $\nu_{Y\bull}$, this means that  $\Gamma(D)_{\nu_1 \ge a}$ is the image of $\Gamma(D - aE)$ under the map
\[  \varphi_a : \NN^d \times \NN \lra \NN^d \times \NN \ \ , \ \ (\nu, m) \mapsto (\nu + ma \Vec{e}_1, m), \]
where as above $\Vec{e}_1 = (1,0, \ldots, 0) \in \NN^d$ is the first standard basis vector. Passing to cones, it follows that 
\[ \Sigma\big( \Gamma(D)_{\nu_1 \ge a} \big) \ = \ \varphi_{a,\RR} \Big( \Sigma\big (\Gamma(D - aE)\big) \Big),\]
where $\varphi_{a, \RR} : \RR^d \times \RR \lra \RR^d \times \RR$ is the evident map on vector spaces determined by $\varphi_a$. This implies that
\[
\Delta(D - aE) \, + \, a \Vec{e_1} \ = \ \Delta(D)_{\nu_1 \ge a},
\]
and hence (upon replacing $D$ by a multiple)
\begin{equation}
\Delta(pD - qE) \, + \, q \Vec{e_1} \ = \ \Delta(pD)_{\nu_1 \ge q} \tag{*}
\end{equation}
whenever $pD - qE$ is big. But both sides of (*) scale linearly, and therefore \eqref{First.Slicing.Equation} holds when both $\xi$ and $t$ are rational. The general case then follows by continuity. 

The proof of the second statement is similar. Specifically, start again with an integral divisor $D$, fix $a > 0$, and denote by 
\[  \Gamma_{X|E}(D-aE) \ \subseteq \ \NN^{d-1} \times \NN \]
the graded semigroup (with respect to the flag $Y\bull|E$) computing the Okounkov body $\Delta_{X|E}(D - aE)$. Then it follows as above from the definition of $\nu_{Y\bull}$ that $\Gamma(D)_{\nu_1 = a} \subseteq \NN^d \times \NN$ coincides with the image of $\Gamma_{X|E}(D - aE)$ under the map
\[
\NN^{d-1} \times \NN \lra \NN^d \times \NN \ \ , \ \ (\nu_2, \ldots, \nu_d, m) \mapsto (m a, \nu_2, \ldots, \nu_d, m). 
 \]
 We assert that
 \begin{equation} \label{diagonal.slice.equation}
 \Sigma \big ( \Gamma(D)_{\nu_1 =a} \big) \ = \ \Sigma\big(\Gamma(D)\big)_{\nu_1 = a},
 \end{equation}
 where the left-hand side denotes the cone generated by the semigroup $\Gamma(D)_{\nu_1 = a}$, and the right-hand side indicates the intersection of $\Sigma\big(\Gamma(D)\big)$ with the evident subspace of $\RR^d \times \RR$. 
Granting this, it follows that $\Delta(D)_{\nu_1 = a} = \Delta_{X|E}(D - aE)$, and hence that $\Delta(pD)_{\nu_1 = q} = \Delta_{X|E}(pD - qE)$ whenever $pD - qE$ is big and $q >0$. As in the previous paragraph, this implies \eqref{Second.Slicing.Equation}. It remains to prove \eqref{diagonal.slice.equation}, but it is a special case of Proposition \ref{Mircea.Prop} from the Appendix.
\end{proof}

\begin{remark} \textbf{(Surfaces.)} The Theorem gives a convenient way to compute $\Delta(D)$ when $\dim X = 2$: see \S \ref{Surface.Subsection}.  \end{remark}


\section{Generic Infinitesimal Flags} \label{Infinitesimal.Constructions.Section}

In this section we study the variation of Okounkov bodies when the relevant data  -- notably the flag $Y\bull$ -- move in flat families. One finds that the resulting body is constant for a very general choice of the parameter. The interest in this is that it allows one to make canonical constructions. Specifically, by working with flags in the exceptional divisor on the blow-up $\Bl_x(X)$ of $X$ at a very general point $x \in X$, one arrives at  Okounkov bodies that do not depend on the arbitrary choice of a global flag on $X$. The exposition here will be a little more condensed than in previous sections.

\subsection{Variation in Families} 
We start by fixing notation. Let
\[ \pi : \mathscr{X} \lra T\] 
be a flat surjective morphism of vareties, of relative dimension $d$, and let $\mathscr{D}$ be a (Cartier) divisor on $\mathscr{X}$ which is flat over $T$. We assume given a flag of subvarieties
\[ \mathscr{Y}\bull \ \ : \ \ \mathscr{X} \ = \ \mathscr{Y}_0 \  \supseteq \ \mathscr{Y}_1 \ \supseteq \ \ldots \ \supseteq \ \mathscr{Y}_d, \]
where $\mathscr{Y}_i$ has codimension $i$ in $\mathscr{X}$, and is flat and surjective over $T$. Denote scheme-theoretic fibres over $t \in T$ with subscripted Roman fonts, so that 
\[ X_t \ = \ \pi^{-1}(t) \ \ , \ \ D_t \ = \ \mathscr{D} | X_t \ \ \text{ and } \ \ Y_{i,t} = \pi^{-1}(t) \cap \mathscr{Y}_i, \] etc. We assume that $T$ is irreducible, and that for every $t \in T$:
\begin{itemize}
\item[(i).]  $X_t$ and each $Y_{i, t}$ are reduced and irreducible;
\item[(ii).] Each $Y_{\bullet, t}$ is an admissible flag on $X_t$.
\end{itemize}
For simplicity, we will also assume
\begin{itemize}
\item[(iii).]  For all $i$,  $\mathscr{Y}_{i+1}$ is a Cartier divisor in $\mathscr{Y}_{i}$ 
\end{itemize}
(and hence the same is true for each $Y_{i+1,t}$ in $Y_{i,t}$).\footnote{Condition (iii) is presumably not necessary for what follows, but it simplifies the discussion and holds in the application we have in mind.}
The Okounkov bodies
\[  \Delta_{Y_{\bullet, t}}\big( X_t; D_t \big) \ \subseteq \ \RR^d \]
are therefore defined provided that $\pi$ is projective.

The main result of this subsection is:
\begin{theorem} \label{Constancy.Of.OB.Thm}
Keeping the notation and hypotheses  just introduced,  assume that $\pi$ is projective and that $D_t$ is big on $X_t$ for all $t \in T$. Then there is a subset 
\[  \mathcal{B} \ = \ \cup  \, B_{m} \ \subset \ T, \] consisting of a countable union of proper Zariski-closed subsets $B_m \subsetneqq T$, such that the Okounkov bodies $ \Delta_{Y_{\bullet, t}} \big( X_t; D_t \big)  $ all coincide for $t \not \in \mathcal{B}$, i.e. 
\[ \Delta_{Y_{\bullet, t}}\big( X_t; D_t \big) \ \subseteq \ \RR^d \ \ \text{ is  independent of $ t$} \] for $t \in T - \mathcal{B}$.
\end{theorem}

\begin{lemma}
Let $\mathscr{E}$ be a Cartier divisor on $X$, flat over $T$, and fix $\sigma \in \ZZ^d$. Then there is a non-empty open subset $U \subseteq T$ such that the dimensions
\[  \dim \, \HH{0}{X_t}{\OO_{X_t}(E_t)}^{\ge \sigma}\]
are constant for $t \in U$, where the group on the right denotes the subspace of $\HH{0}{X_t}{\OO_{X_t}(E_t)}$ consisting of sections having valuation $\ge \sigma$ with respect to $\nu_{Y_{\bullet, t}}$.
\end{lemma}

\begin{proof}
Write $\mathscr{L} = \OO_{\mathscr{X}}(\mathscr{E})$ and $L_t = \mathscr{L} | X_t$. Viewing $\mathscr{Y}\bull$ as a partial flag on $\mathscr{X}$, denote by $\mathscr{L}^{\ge \sigma}\subseteq \mathscr{L}$ the subsheaf of $\mathscr{L}$ introduced in Remark \ref{sheafification.remark}. It follows from the construction that $\mathscr{L}^{\ge \sigma}$ is flat over $T$, and that 
\[
\mathscr{L}^{\ge \sigma} \otimes \OO_{X_t} \ = \ \big( L_t\big) ^{\ge \sigma}
\]
for all $t \in T$.\footnote{Thanks to our simplifying hypothesis (iii), the sheaves in question are computed globally on $\mathscr{X}$ and $X_t$ by the iterative procedure in Remark \ref{sheafification.remark}, without requiring recourse to equation \eqref{Close.Up.Sheaf.Eqn}.
}
 Since
\[ \HH{0}{X_t}{L_t^{\ge \sigma}} \ = \ \HH{0}{X_t}{L_t}^{\ge \sigma}, \]
the assertion of the Lemma follows from the semicontinuity theorem.\end{proof}

\begin{proof}[Proof of Theorem \ref{Constancy.Of.OB.Thm}] Fix $m \ge 0$, and consider the maps
\begin{equation}
\nu_{Y_{\bullet,t}} \, : \, \big( \HH{0}{X_t}{\OO_{X_t}(mD_t)} \, - \{ 0 \} \,   \big) 
 \ \lra \ \ZZ^d.
\notag
\end{equation}
It is enough to show that there is a non-empty open set $U_m $ such that the image of $\nu_{Y_{\bullet,t}}$ is independent of $t$ for $t \in U_m$, for then one can take  $B_m = T - U_m$. To this end, note first that there is an open set $U_m^\pr \subseteq T$  on which the dimension of the groups $\HH{0}{X_t}{\OO_X(mD_t)}$ is constant. Thus the images of $\nu_{Y_{\bullet, t}}$ for $t \in U_m^\pr$ -- which are ordered subsets of $\ZZ^d$ -- all have the same cardinality. Furthermore, it follows from the proof of Lemma \ref{boundedness}
 that these images all lie in a fixed finite subset of $\ZZ^d$. The Theorem then follows by applying the previous Lemma to the elements of this finite set.
 \end{proof}

 \subsection{Infinitesimal Okounkov Bodies} 
 \label{Infinitesimal.OB.Subsection}
 We now indicate how the results of the previous subsection lead to the possibility of eliminating the choice of a fixed global flag on $X$. The idea is to use infinitesimal data -- which automatically vary in families -- to get a flag on the blow-up of $X$ at a very general point. 
 
 As usual, let $X$ be an irreducible projective variety of dimension $d$. Fix a smooth point $x \in X$, as well as a complete flag $V\bull$ of subspaces
 \[
 T_xX \ = \ V_0 \ \supseteq \ V_1 \ \supseteq \ V_2 \ \supseteq \ \ldots \ \supseteq \ V_{d-1} \ \supseteq \ \{ 0 \}
 \] 
 in the tangent space to $X$ at $x$. Consider the blowing up 
 \[  \mu : X^\pr \, = \, \Bl_x(X) \lra X 
 \] of $X$ at $x$, with exceptional divisor $E$. The projectivizations of the $V_i$ give rise to a flag $F\bull = F(x; V\bull)$ in $X^\pr$:
 \[
 X^\pr \ \supseteq \ E = \PPP(T_xX) \ \supseteq \ \PPP(V_1) \  \supseteq \ \PPP(V_2) \ \supseteq \ \ldots \ \PPP(V_{d-1})  \, = \, \{ \text{pt} \}.
 \]
 
 On the other hand, let $D$ be any divisor on $X$, and write $D^\pr = \mu^* D$.  Then
 \[ \HH{0}{X}{\OO_X(mD)} = \HH{0}{X^\pr}{\OO_{X^\pr}(mD^\pr)} \]
 for all $m$. Therefore the choice of a flag $Y^\pr\bull$ on $X^\pr$ determines a valuation also on sections of $D$, and we write $\Delta_{Y^\pr \bull}(D)$ for the corresponding convex body. In other words,  
 \[
 \Delta_{Y^\pr \bull} (D) \ =_{\text{def}} \ \Delta_{Y^\pr\bull}(D^\pr),
 \]
 the object on the right being constructed on $X^\pr$.
 In particular, the choice of a flag $F\bull$ as above gives rise to a convex body $\Delta_{F\bull}(D) \subseteq \RR^d$. As in \S \ref{Glob.Ok.Body.Subsection}, these occur as the fibres of closed convex cones 
 \[  \Delta_{F\bull}(X) \subseteq \ \RR^d \times N^1(X)_{\RR}.
\]
 
 \begin{proposition} \label{Generic.Infinitesimal.Body}
 Let $D$ be any big divisor on $X$. Then the corresponding Okounkov bodies
 \[
 \Delta_{F(x, V\bull) } (D) \ \subseteq \ \RR^d 
 \]
 all coincide for a very general choice of $x \in X$ and the flag $V\bull$. The analogous statement holds for the global bodies $\Delta_{(F, V\bull)}(X)$. 
  \end{proposition}
  
  \begin{definition}
  We denote by \[ \Delta^\pr(D) \subseteq \RR^d   \ \ , \ \  \Delta^{\pr}(X) \, \subseteq \, \RR^d \times N^1(X)_{\RR} \]
  the sets $ \Delta_{F(x, V\bull) }(D)$ and  $ \Delta_{F(x, V\bull) }(X)$ for very general choices of $F(x, V\bull)$. \qed
    \end{definition}
 
 \begin{proof}[Proof of Proposition \ref{Generic.Infinitesimal.Body}] The first assertion  follows immediately from Theorem \ref{Constancy.Of.OB.Thm} since the data at hand move in an  algebraic familiy parametrized by a suitable open subset of the manifold of full flags in the tangent bundle to the smooth locus of $X$. We can assume moreover that the statement holds simultaneously for countably many big divisors $D$ on $X$, and then the assertion for the global Okounkov bodies follows. 
 \end{proof}
 
 \begin{remark} It is interesting to ask what geometric information these convex bodies encode.
 One can show using Theorem \ref{Slicing.Theorem}
and the results of \cite{DeFKL} that  $\Delta^\pr(D)$ determines the Seshadri constant $\eps(D;x)$ of an ample divisor $D$ at a very general point of $x$. The well-known difficulty of calculating these invariants reinforces our own experience that the convex bodies $\Delta^\pr(D)$ are in general very hard to compute. (See   \cite[Chapter 5]{PAG} for an overview of Seshadri constants.)  \end{remark}

 \section{Examples} \label{Examples.Section}
 This section is devoted to some examples and computations. We start with toric varieties. In the second subsection we describe completely the Okounkov body of a big divisor on a smooth surface.  Finally, in \S \ref{Round.Cone.Ex} we give an example to show that $\Delta(D)$ need not be polyhedral.  For simplicity, we work here over the complex numbers $\CC$.
 
 \subsection{Toric Varieties}
 We show that on a smooth toric variety, the Okounkov construction recovers the familiar correspondence between divisor classes and lattice polytopes.
 
We start by fixing some notation.  Let $X$ be a $d$-dimensional smooth projective toric variety,  corresponding to a fan  in $N_{\RR}\simeq\RR^d$, so that the torus $T=N\otimes_{\ZZ}\KK^*$ acts on $X$. Let 
$D$ be a $T$-invariant divisor on $X$
(see \S 3.4 in \cite{Fulton}  for  notation and basic facts about divisors on toric varieties). 
Every lattice point in the dual space $M_{\RR} = N_{\RR}^*$ determines a rational function
$\chi^u$ on $X$. 
One associates to $D$ a polytope $P_D$ in 
$M_{\RR}$, such that the lattice points in $P_D$ are those $u\in M$ with
$D+{\rm div}(\chi^u)\geq 0$. In this way, $P_D\cap M$ gives a basis of isotypical sections
 of  $H^0(X,\cO(D))$.
If we replace $D$ by a linearly equivalent divisor, then the polytope changes accordingly: 
$P_{D+{\rm div}(\chi^w)}=P_D-w$.
Moreover, we have $P_{mD}=mP_D$
for every positive integer $m$, which allows us to define in the obvious way $P_D$
when $D$ is an invariant $\QQ$-divisor. 

Suppose that the flag $Y_{\bull}$ consists of $T$-invariant subvarieties of $X$.
Since $X$ is smooth, we can order the prime $T$-invariant divisors $D_1,\ldots,D_s$ of $X$ such that
$Y_i=D_1\cap\ldots\cap D_i$, for $i\leq d$. 
If we denote by $v_i$ the primitive generator of the ray corresponding to $D_i$, then 
$v_1,\ldots,v_d$
form a basis of $N$, and they generate a maximal cone $\sigma$ in the fan of $X$.
We get an isomorphism $\ZZ^d\simeq N$, and a dual isomorphism $\phi\colon M\to\ZZ^d$,
 given by $\phi(u)=(\langle u,v_i\rangle)_{1\leq i\leq d}$, which in turn determines a linear map
$\phi_\RR:M_{\RR} \overset{\cong} \lra \RR^d.$
 
 On every smooth toric variety there is an exact sequence
 $$0\lra M\overset{\iota}\lra\ZZ^s\overset{q}\lra{\rm Pic}(X)\lra 0,$$
 and ${\rm Pic}(X) = N^1(X)$ has no torsion.
If we identify $\ZZ^s$ with the group of $T$-invariant divisors, then $q$ is the map taking a divisor to its class,
and $ \iota(u)={\rm div}(\chi^u)=\sum_{i=1}^s\langle u,v_i\rangle D_i.$
The above choice of a basis for $N$ induces a splitting of this exact sequence, and
consequently an isomorphism
\[ \psi\, \colon \, \ZZ^d\times {\rm Pic}(X)\lra \ZZ^s,\] such that 
$\psi^{-1}(D)=(p(D), q(D))$, $p\colon\ZZ^s\lra\ZZ^d$
being the projection onto the first $d$ components. 

\begin{proposition} \label{Toric.Example.Prop}
Let $X$ be a smooth projective toric variety, and let $Y_{\bull}$ be a flag
of invariant subvarieties chosen as above. 
\begin{enumerate}
\item[(i).] Given any big line bundle $L$ on $X$, let $D$ be the unique $T$-invariant divisor such that
$L\simeq\cO_X(D)$ and $D\vert_{U_{\sigma}}=0$. Then\[ \Delta(L)\ =\ \phi_{\RR}(P_D). \]
\lbl
\item[(ii).] The global Okounkov body $\Delta(X)$ is the inverse image under the isomorphism
 \[ \psi_{\RR} \ : \ \RR^d \times N^1(X)_{\RR} \overset{\cong} \lra \RR^s \]
of the non-negative orthant $\RR_+^s \subseteq \RR^s$.
\end{enumerate}
\end{proposition}

\begin{remark}
The statement in (ii) was pointed out to us by Diane Maclagan. \qed
\end{remark}
\begin{proof}
Since $X$ is smooth, the divisor
 $\sum_{i=1}^s D_i$ has simple normal crossings.
It follows that if $s\in\HH{0}{X}{L}$ is a section with zero locus $\sum_{i=1}^sa_iD_i$, then 
\[ \nu_{Y_{\bull}}(s)\ =\ (a_1,\ldots,a_d). \]
Now consider a lattice point $u\in P_D$. Then the zero-locus of the corresponding section $\chi^u\in \HH{0}{X}{\OO_X(D)} $ is
$D+\sum_{i=1}^r\langle u,v_i\rangle D_i$. By assumption, $D\vert_{U_{\sigma}}=0$, hence
$\nu_{Y_{\bull}}(\chi^u)=\phi(u)$. Since $\phi$ is injective, and we have precisely 
$h^0(L)$ lattice points in $P_D\cap M$, it follows that  \[ \Image \left( \, \big( \HH{0}{X}{L} - \{0 \} \big) \overset{\nu_{Y\bull}}\lra \ZZ^d  \right) \ = \ \phi(P_D\cap M). \]
But now suppose that $m$ is any positive integer such that $mP_D$ has all its vertices in the lattice. In this case,
the convex hull of $\frac{1}{m}\phi(mP_D\cap M)$ is equal to $\phi_{\RR}(P_D)$. In particular, it is independent of $m$, and we conclude that $\phi_{\RR}(P_D)=\Delta(D)$. 
This shows i), and in fact, we get the same assertion for every element in ${\rm Pic}(X)_{\QQ}$.

Consider now the semigroup $S$ in $M\times {\rm Pic}(X)$ consisting of pairs $(u,L)$ with the property
 that if $D$ is the unique $T$-invariant divisor with $\cO(D)\simeq L$ and $D\vert_{U_{\sigma}}=0$,
 then $u\in P_D$. It follows from what we showed so far that in order to prove ii), it is enough to show
 that $\Phi(S)= \NN^s$, where $\Phi\colon M\times{\rm Pic}(X)\to\ZZ^s$ is the isomorphism $\psi\circ (\phi, {\rm Id})$.  If we identify $\ZZ^d$ with the group of $T$-invariant divisors, then $\Phi^{-1}(E)=(u,[E])$,
 where $u\in M$ is such that $E\vert_{U_{\sigma}}={\rm div}(\chi^u)\vert_{U_{\sigma}}$. 
 We have $(u,[E])\in S$ if and only if $u\in P_{E-{\rm div}(\chi^u)}=u+P_E$. This is the case if and only if
 $0\in P_E$, that is, $E$ is effective. Hence $\Phi^{-1}(E)\in S$ if and only if $E\in\NN^s$.
\end{proof}

\subsection{Surfaces} \label{Surface.Subsection}
In this subsection, we use the results of \S \ref{Slices.Subsection} to 
describe the Okounkov body of any big divisor on a surface.

Let $X$ be a smooth complex projective surface, and let $D$ be a big $\QQ$-divisor on $X$. Recall that any such divisor has a  \textit{Zariski decomposition}     $D=P+N$,  where $P$ (the positive part of $D$) is a nef $\QQ$-divisor, and $N $ 
(the negative part of $D$)  is an effective $\QQ$-divisor, with the property
such that whenever $mD$ and $mN$ are integral divisors, multiplication by the section defining $mN$ induces an isomorphism
\begin{equation} \label{Zariski.Decomp.Equation}
H^0(X,\cO(mP))\overset{\cong} \lra H^0(X,\cO(D)).
\end{equation}
The key point for us is that inside the big cone the Zariski decomposition varies in a piecewise linear way. More precisely, there are disjoint open convex subcones ${\mathcal C}_i \subseteq \tn{Big}(X)$ of the big cone with
the following properties:
\begin{enumerate}
\item[i)] For every index $i$ there are irreducible curves $T_1,\ldots,T_r$ such that for all big divisors
$D\in \overline{{\mathcal C}_i}$, the negative part of $D$ is supported on $T_1\cup\ldots\cup T_r$, and the map taking $D$ to its negative part is linear on the intersection of $\overline{{\mathcal C}_i}$
with the big cone.
\lbl
\item[ii)] Around every point in the big cone, each $\overline{{\mathcal C}_i}$ is rational and polyhedral,
and there are only finitely many such cones.  
\end{enumerate}
For a proof and details, see \cite{BKS} or 
\cite[Example~3.7]{ELMNP3}. It follows from these linearity properties that the Zariski decomposition $D = P + N$ is naturally defined for any big $\RR$-divisor. 

Fix henceforth an admissible flag
\begin{equation} \label{Flag.For.Sf} X \ \supseteq \ C \ \supseteq \ \{x\}, \end{equation}
on $X$, where $C \subseteq X$ is an irreducible curve and $x \in C$ is a smooth point. The first thing to note is that the Zariski decomposition of a divisor $D$ determines the Okounkov body of the restricted linear series  of $D$ from $X$ to $C$:\begin{lemma} \label{Sfs.Ex.Lemma}
Let $D$ be a big $\QQ$-divisor on $X$ with Zariski decomposition $D = P + N$. Assume  that $C \not \subseteq \BBB_+(D)$, so that in particular  $C \not \subset \Supp(N)$, and set
\[ \alpha(D) \ = \ \ord_x \big( N_{|C} \big ) \ \ , \ \ \beta(D) \ = \ \ord_x \big( N_{|C} \big ) \ + \ \big( C \cdot P \big). 
\]
Then the Okounkov body of the restricted complete linear series of $D$ is the interval
\[
\Delta_{X|C}(D) \ = \ [\, \alpha(D) \ , \    \beta(D) \, ] \ \subseteq \ \RR.\]
\end{lemma}
\begin{proof}[Sketch of Proof] Recalling  from \cite{ELMNP3}  that $\vol_{X|C}(D) = \big( C \cdot P \big) $, this follows easily from \eqref{Zariski.Decomp.Equation} and Example \ref{OB.Curves.Ex}. 
\end{proof}

Consider  now   a big $\QQ$-divisor $D$, and write
\[
\mu = \mu(D; C) \ = \ \sup \{ \, s > 0 \mid D - sC \text{ is big } \}.
\]
\begin{theorem}\label{surface_example}
With the above notation, there are continuous functions 
\[
\alpha \ , \ \beta\, : \, [a, \mu] \lra \RR_+
\]
for some $0 \le a \le \mu$, with $\alpha$ convex, $\beta$ concave, and  $\alpha\leq\beta$, such that $\Delta(D) \subseteq \RR^2$ is the region bounded by the graphs of $\alpha$ and $\beta$:
$$\Delta(D)=\big \{\, (t,y)\in\RR^2\mid a\leq t\leq \mu,\ \text{and} \ \alpha(t)\leq y\leq\beta(t)\, \big \}.$$
 Moreover, both $\alpha$ and $\beta$ are   piecewise linear and rational on every
interval $[a,\mu']$  with $\mu'<\mu$. In particular, the intersection of $\Delta(D)$ with 
$[0,\mu']\times\RR$ is a rational polytope.
\end{theorem}

\noi The theorem is illustrated schematically in Figure \ref{Surfaces.Picture}.

 \begin{figure}\vskip 5pt
\includegraphics[scale = .6]{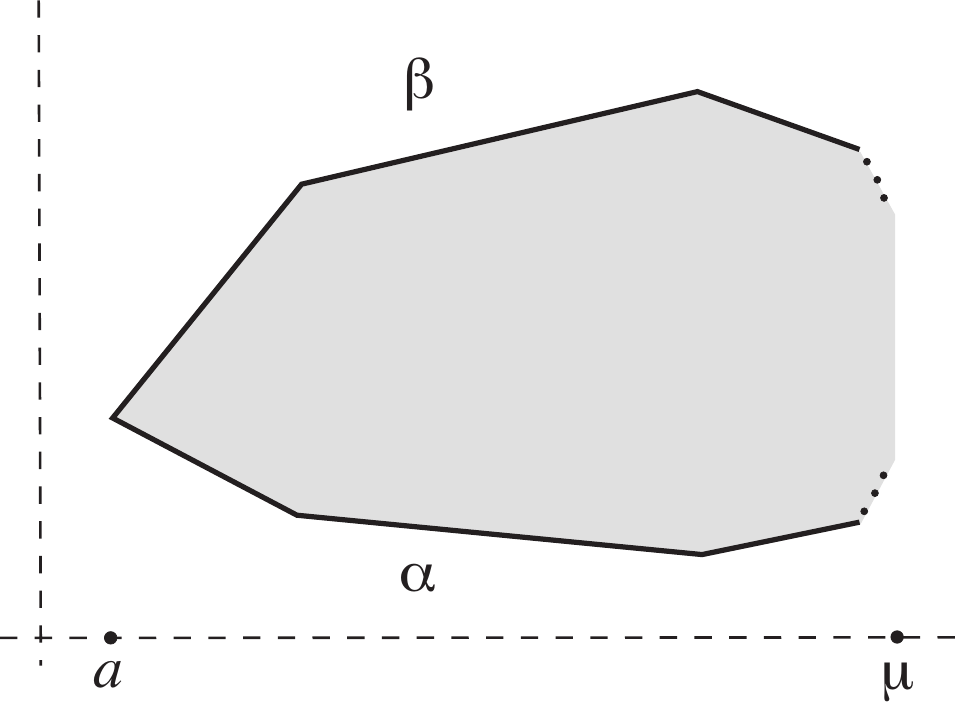}
\caption{Okounkov body of a divisor on  a surface}
 \label{Surfaces.Picture}
\end{figure}

\begin{proof}[Proof of Theorem \ref{surface_example}] For $t\in [0,\mu)$, put $D_t = D - tC$, and write $D_t = P_t + N_t$ for its Zariski decomposition. 
Let $a$ be the coefficient of $C$ in $N_0$. Since  $D-aC$ is big, and since $\Delta(D)=\Delta(D-aC)+(a,0)$,
we may replace $D$ by $D-aC$. Therefore we may suppose  that $C$ does not appear in $N_0$.  Note that in this case $C$ does not appear in the support of any $N_t$, with $t<\mu$.  

Let
\[
\alpha(t) \ = \ \ord_x(N_t{_{|C}}) \ \ , \ \ \beta(t) \ = \  \ord_x(N_t{_{|C}}) \, + \, \big( C \cdot P_t)
\]
be the two quantities appearing in the statement of Lemma \ref{Sfs.Ex.Lemma}. It follows from Theorem \ref{Slicing.Theorem} and the Lemma that $\Delta(D)$ is the region bounded by the graphs of $\alpha(t)$ and $\beta(t)$. The fact that $\alpha$ is convex, and $\beta$ is concave is a consequence of the convexity
of $\Delta(D)$. If we put $\alpha(\mu):=\min\{y\geq 0\mid (\mu,y)\in\Delta(D)\}$, and
$\beta(\mu):=\max\{y\geq 0\mid (\mu,y)\in\Delta(D)\}$, then $\alpha$ and $\beta$ are continuous
on $[0,\mu]$. The piecewise linearity properties of $\alpha$ and $\beta$ follow from the facts quoted at the beginning of this subsection concerning the variation of Zariski decomposition. \end{proof}

\begin{example}\textbf{(Abelian surfaces.)} The example of a divisor $D$ on an abelian surface considered in the Introduction  (see Figure \ref{Okounkov.Body.Picture}) follows at once from the theorem. In this case $D_t = D  - tC$ is nef for all $t \le \mu(D)$, so the negative part of the Zariski decomposition does not occur. \qed
\end{example}

 This picture extends to describe 
the global Okounkov body $\Delta(X)$. In particular, the decomposition of the big cone
induced by the cones ${\mathcal C}_i$ gives the following corollary.

\begin{corollary}\label{global_body_surface}
Let $X$ be a smooth complex projective surface. Fix a flag \eqref{Flag.For.Sf}, and let $c \in N^1(X)$ denote the class of $C$. If \[ \Delta(X)\ \subseteq 
\ N^1(X)_{\RR}\times \RR^2\] is the corresponding global Okounkov body of $X$,
then
$$\Delta(X)\subseteq \{(\xi,t,y)\mid \xi-tc\in\Effbar(X)\},$$
and $\Delta(X)$ is a rational polytope in the neighborhood of
  every point $(\xi,t,y)$ with  $\xi-tc\,\text{big}$.
\end{corollary}

\subsection{A Non-Polyhedral Okounkov Body} 
\label{Round.Cone.Ex} We establish the existence of a big divisor $D$ on a fourfold $X$ for which $\Delta(D)$ is not polyhedral. The idea is to use a construction of Cutkosky, as explained in \cite[Chapter 2.3]{PAG}.

Let $V$ be an abelian surface having Picard number $\rho(V) = 3$, so that $\tn{Nef}(V)  = \overline{\tn{Eff}}(V)$ is a circular cone in $\RR^3$.  Choose ample divisors $A, B_1, B_2$ on $V$, and let
\[  \mathscr{E} \ = \ \OO_V(A_1) \oplus \OO_V(-B_1) \oplus \OO_V(-B_2). \]
Put $X = \PP(\mathscr{E})$, with $\pi : \PP(\mathscr{E}) \lra V$ the bundle map. For $D$ we take a divisor on $X$ such that $\OO_X(D) = \OO_{\PP(\mathscr{E})}(1)$, and we consider a flag $Y\bull$ where
\[  Y_1 \ = \ \PP \big( \OO_V(A) \oplus \OO_V(-B_1)\big) \ \ , \ \ Y_2 \ = \ \PP \big( \OO_V(A) \big) . \]
We assert that the corresponding Okounkov body $\Delta(D) \subseteq \RR^4$ cannot be polyhedral provided that we make suitably general choices of $A, B_1$ and $B_2$.

To see this, we consider slices of $\Delta(D)$ as in \S \ref{Slices.Subsection}.  Specifically, observe first that $\OO_X(Y_1) = \OO_X(D + \pi^* B_2)$. As we are dealing with decomposable projective bundles, one finds that the restriction maps
\[ 
\HH{0}{X}{\OO_X(pD - qY_1)} \lra \HH{0}{Y_1}{\OO_{Y_1}(pD - qY_1)} \]
are surjective for all $p, q > 0$. This implies in particular that
\begin{equation}
\Delta_{X|Y_1}(D-tY_1) \ = \ \Delta\big(Y_1; (D - tY_1)_{|Y_1}\big) \tag{*}
\end{equation}
for all $0 \le t < 1$. 
Now assume that $\Delta(D)$ is polyhedral. Then so too are all of its slices $\Delta_{X|Y_1}(D - tY_1)$, and moreover these vary piecewise linearly with $t$.  More precisely,  the invariant
\[
\mu(t) \ =_{\text{def}} \ \sup \, \big\{ s > 0 \mid (D-tY_1)_{|Y_1} - sY_2 \ \text{ is a big divisor on $Y_1$ } \big \}
\]
measuring the right-hand endpont of $\Delta\big(Y_1; (D - tY_1)_{|Y_1}\big) $ under projection to the $\nu_2$-axis  must vary linearly with $t$ for small $t > 0$. On the other hand, note that \begin{multline*}
\HH{0}{Y_1}{\OO_{Y_1}( (pD - qY_1)_{|Y_1} - rY_2)} \\
= \ \HHH{0}{V}{S^{p-q-r}\big( \OO_V(A) \oplus \OO_V(-B_1)\big ) \otimes \OO_V(-qB_2 - rB_1) }.
\end{multline*}
Therefore 
\begin{align*}\mu(t) \ &= \ \sup \, \big \{ s > 0 \mid (1 - t -s)A - tB_2 - sB_1 \text{ is a big divisor on $V$ } \big \} \\
&=  \ \sup \, \big \{ s > 0 \mid \big( (1 - t -s)A - tB_2 - sB_1\big)^2   \ge 0 \big \}.
\end{align*}
But as in \cite[Chapter 2.3B]{PAG}, for general choices of $A, B_1, B_2$ this is a non-linear function of $t$ since the pseudo-effective cone of $V$ is circular.

  \section{Questions and Open Problems} \label{Open.Problem.Section}
 
 We pose  here a few questions and open problems.

It is natural to ask whether the constructions we study here behave particularly well on special classes of varieties. For instance, if $X$ is a smooth toric variety, then  as we have seen  $\Delta(X)$ is a polytope.
One of the consequences of the spectacular recent progress \cite{BCHMcK} on the minimal model program is that linear series on a smooth Fano variety have many toric-like features. This suggests:
 \begin{problem} \label{OK.Body.Fanos}
 Let $X$ be a smooth complex projective Fano variety. Does there exist an admissible flag on $X$ with respect to which $\Delta(X)$ is a rational polyhedral cone?
 \end{problem}
 \noi This would of course imply that $\Delta(D)$ is polyhedral for every big divisor on $D$. More generally, one could ask the same question for the ``Mori dream spaces" studied by Hu and Keel in \cite{HK}. Knowing that $\Delta(X)$ is polyhedral would for example recover the fact that the volume function $\vol_X$ is piecewise polynomial for Fanos (and for Mori dream spaces in general).
 
The nature of the volume function $\vol_X: N^1(X)_{\RR} \lra \RR$ on a smooth projective variety presents several intriguing questions. Corollary \ref{Continuity.Abstract.Volume.Function}
 shows that the continuity and log-concavity of this function are quite formal -- i.e. they hold already for multi-graded linear series -- and presumably the differentiability properties established in \S \ref{Slices.Subsection}
 could also be formalized with some additional hypotheses. By analogy with  the work \cite{Wolfe} of Wolfe on a related invariant, one nonetheless expects that there exist  multi-graded linear series $W_{\Vec \bullet}$ for which the corresponding volume function $\vol_{W_{\Vec{\bullet}}}$ is rather wild:
 \begin{problem}
 Construct examples of multi-graded linear series for which $\vol_{W_{\Vec{\bullet}}}$ is nowhere $\mathcal{C}^\infty$ (or even $\mathcal{C}^3$) on an open set.
 \end{problem}
 \noi On the other hand, it is difficult to imagine that this sort of behavior can occur for the volume function $\vol_X$ on a projective variety $X$.  So one anticipates that $\vol_X$ should have some good properties beyond those already known, but it has not been clear how to  make plausible conjectures about what these might be. One possible approach is to look for special features of the global Okounkov body $\Delta(X)$ (with respect to a suitable flag), the hope being that geometric information about $\Delta(X)$ is a natural way to express regularity properties of the volume function. 
\begin{question}
 What can one say about the boundary of $\Delta(X)$? Is it for example ``almost everywhere" defined by algebraic equations?
 \end{question}
  \noi The issue we have in mind here is whether there is an analogue of a theorem of Campana and Peternell (\cite{CP}, \cite[Chapter 1.5.E]{PAG}) according to which the boundary of the nef cone on a projective variety is generically defined by algebraic hypersurfaces. At the moment, of course, this question is completely speculative: we know of very few examples where one can actually compute $\Delta(X)$, and so up to now there is very little evidence one way or the other.   
  
  The canonical Okounkov bodies $\Delta^\pr(D)$ and $\Delta^\pr(X)$ appearing in \S \ref{Infinitesimal.OB.Subsection}  also  merit further investigation. The most important invariant of a convex body $K \subseteq \RR^d$ is its volume, but convex geometers have studied many other invariants as well, for example the Minkowski surface area, and the sequence of intrinsic volumes of $K$ (cf \cite[Chapter 6.3]{Gruber}). As the convex bodies $\Delta^\pr(D)$ are intrinsically defined, it seems reasonable to pose
      \begin{problem}
     Find algebro-geometric interpretations of convex-geometric invariants of $\Delta^\pr(D)$.
  \end{problem}  
\noi Unfortunately, the bodies $\Delta^\pr(D)$  seem very hard to compute --  for instance,  we do not know how to describe  them already when $D$ is an ample divisor on   the product   $X = \PP^1 \times \ldots \times \PP^1$  of $d\ge 4$ copies of $\PP^1$.   This suggests:
  \begin{problem}
  Are there other constructions that lead to canonically defined Okounkov bodies that are more amenable to computation?
  \end{problem}
  \noi It seems likely that one could globalize the description in \cite{BayerMumford} of the reverse lex order on polynomials, although it is not clear whether the resulting Okounkov bodies will be much more tractable. 
    
  Finally, asymptotic invariants of linear series have appeared in other settings, and it is natural to wonder whether the machinery developed here extends as well.    Paoletti and others \cite{Paoletti04}, \cite{Paoletti05}, \cite{Vedova.Paoletti} have studied equivariant volume functions and related invariants in the presence of a group action. This suggests:
  \begin{problem}
  Extend the theory in the present paper to an equivariant setting.
  \end{problem}
  \noi The original paper \cite{Okounkov96} of Okounkov, as well as \cite{Okounkov97}, \cite{Alexeev.Brion}, \cite{KK}, \cite{Kaveh},  might be relevant. 
 There has also been some very interesting recent work on arithmetic analogues of the volume function \cite{Yuan}, \cite{Moriwaki}, which leads to:
  \begin{question}
  Can one construct ``arithmetic Okounkov bodies"?
  \end{question}
  \noi When $X$ is a compact complex manifold, Boucksom \cite{Boucksom}, \cite{Boucksom04} has defined and studied the volume (and other invariants) of an arbitrary pseudo-effective $(1,1)$-class $\alpha$ on $X$. It is natural to wonder whether one can realize these volumes by convex bodies  as well. 
  

  \appendix
  \section{Semigroups and Subspaces} \label{Appendix.Section}
  
  We prove here a result on the relation between   semigroups and the cones they span upon intersecting with a subspace.

  Let $\Gamma\subseteq\NN^n$ be a sub-semigroup, and denote by  
  \[ \Sigma\ = \ \Sigma(\Gamma)\ \subseteq \
\RR^n \]
 be the closed convex cone generated by $\Gamma$. 
 Given a linear subspace $L \subseteq \RR^n$ defined over $\QQ$, 
 we may intersect $\Gamma $ with $L$ to get a semigroup $\Gamma \cap L  \subseteq L$, which in turn determines a cone $\Sigma(\Gamma \cap L) \subseteq L$. On the other hand, we may intersect $\Sigma(\Gamma)$ with $L$ to get another cone in $L$. We seek conditions under which  these two cones coincide.
 
 \begin{proposition} \label{Mircea.Prop}
 Assume that $\Gamma$ generates a subgroup of finite index in $\ZZ^n$, and that $L$ meets the interior $\interior(\Sigma)$ of $\Sigma$. Then
 \[
 \Sigma(\Gamma) \, \cap \, L \ = \ \Sigma (\Gamma \cap L).
 \]
 \end{proposition}
 \noi Note that the hypothesis on $\Gamma$ is equivalent to asking that $\Sigma$ be full-dimensional, i.e. that $\Sigma$ has non-empty interior in $\RR^n$.
 
The plan is to approximate $\Gamma$ by finitely generated semigroups.  So fix a sequence of finitely generated 
sub-semigroups $\Gamma^1\subseteq\Gamma^2\subseteq\ldots\subseteq\Gamma$, each generating a subgroup of finite index in $\ZZ^n$, such that
$\Gamma=\cup_i\Gamma^i$. Let $\Sigma^i=\Sigma(\Gamma^i)\subseteq\RR^n$ be the corresponding cones. Evidently $\Sigma=\overline{\cup_i\Sigma^i}$.

\begin{lemma}
One has ${\rm int}(\Sigma)=\bigcup_i{\rm int}(\Sigma^i)$. 
\end{lemma}

\begin{proof}
If $\gamma\in {\rm int}(\Sigma)$,
choose linearly independent $v_1,\ldots,v_n$ such that
$\gamma$ lies in the interior of the convex cone generated by the $v_i$. 
Fix $m\gg 0$, such that each
 $\gamma+\frac{1}{m}v_j\in {\rm int}(\Sigma)$ for every $j$, and so that in addition these $n$
 vectors are linearly independent.
 In this case, $\gamma$ lies in the interior of the convex cone
 generated by 
 \[ \gamma+\frac{1}{m}v_1\ , \ \ldots\ , \ \gamma+\frac{1}{m}v_n. \] Furthermore, if $w_j$ is close enough to
 $\gamma+\frac{1}{m}v_j$, then $\gamma$ lies in the interior of the convex cone generated by
 $\{w_1,\ldots,w_n\}$. We can find such $w_1,\ldots,w_n$ and $i$ with 
 $w_j\in \Sigma^i$ for all $j$. Therefore $\gamma\in {\rm int}(\Sigma^i)$. This proves the lemma, since the reverse
 inclusion is trivial.
\end{proof}

\begin{proof}  [Proof of Proposition \ref{Mircea.Prop}]
It is enough to prove the inclusion $\Sigma\cap L\subseteq\Sigma(\Gamma\cap L)$, as the 
reverse inclusion is clear. Suppose that  $\gamma\in \Sigma\cap L$. By assumption, we can choose a vector
$\gamma_0\in {\rm int}(\Sigma)\cap L$. Since the line segment $[\gamma_0,\gamma)$
is contained in ${\rm int}(\Sigma)\cap L$, and since it is enough to show that this segment is contained 
in $\Sigma(\Gamma\cap L)$, we may assume that $\gamma\in {\rm int}(\Sigma)\cap L$.

It follows from the Lemma that   $\gamma\in {\rm int}(\Sigma^i)$ for some $i$. So after replacing
$\Gamma$ by $\Gamma^i$,  we may assume that $\Gamma$ is finitely generated.  In this case,
$\Sigma$ and $\Sigma\cap L$ are rational polyhedral cones. In particular, $\Sigma\cap L$
is the convex cone generated by the semigroup $\Sigma\cap L\cap\ZZ^n$. Furthermore, given any
$\delta\in\Sigma\cap\ZZ^n$, there is $m\geq 1$ such that $m\delta\in\Gamma$. (See (0.5) in \S \ref{Notation.Conventions.Section}). In particular,
$\Gamma\cap L$ and $\Sigma\cap\ZZ^n\cap L$ generate the same convex cone, which completes the proof.
\end{proof}

\begin{corollary} \label{Mircea.Corollary}
Keep the assumptions of the Proposition, and consider for 
$m\leq n$ the projection  $p\colon \RR^n\to \RR^m$  onto the last $m$ components. Let
$L\subseteq \RR^m$ be a linear subspace defined over $\QQ$
such that ${\rm int}(\overline{p(\Sigma)})\cap L\neq\emptyset$. Then $$\Sigma(\Gamma)\ \cap \ p^{-1}(L)=\Sigma(\Gamma \, \cap \, p^{-1}(L)).$$
\end{corollary}

\begin{proof}
By assumption we can find  
$\delta\in {\rm int}(\overline{p(\Sigma)})\cap L$, and the assertion will
follow from the Proposition if we show that
\[ p^{-1}(\RR\cdot \delta) \ \cap\  {\rm int}(\Sigma)\neq\emptyset. \]
Thanks to the Lemma, we may replace $\Sigma$ by one of the $\Sigma^i$, and hence assume that
$\Sigma$ is polyhedral. 

By the choice of $\delta$, the intersection 
 $p^{-1}(\RR \cdot \delta)\cap \Sigma$ is nonempty. If $p^{-1}(\RR\cdot \delta)$ does not meet the interior
 of $\Sigma$, then it is contained in one of the faces of $\Sigma$ (here we use the fact that
 $\Sigma$ has full dimension). In this case we can find a linear function $\ell$ on $\RR^n$ that is nonnegative
 on $\Sigma$ and vanishes on $p^{-1}(\RR\cdot \delta)$ such that 
 $$p^{-1}(\RR\cdot \delta)\, \cap\,  \Sigma \ \subseteq \ \Sigma\cap (\ell=0).$$
 We get an induced linear function $\overline{\ell}$ on $\RR^m$ such that $\ell=\overline{\ell}
 \circ p$. Since $\overline{\ell}$ is nonnegative on $p(\Sigma)$, and  vanishes on $\delta$,
 this contradicts the fact that $\delta\in {\rm int}(p(\Sigma))$. 
\end{proof}


\end{document}